\documentclass[reqno,10pt,a4paper,twoside]{amsart}
\usepackage{amsfonts, amsmath, amscd, latexsym, mathabx, amssymb}%
\usepackage[dvips]{graphics}
\usepackage{pst-all}
\usepackage{float}
\usepackage{mathtools}
\DeclarePairedDelimiter{\ceil}{\lceil}{\rceil}

\usepackage{enumerate}
\usepackage{enumitem}
\usepackage{subfigure}
\usepackage{mathrsfs}
\usepackage{multirow, pifont}
\usepackage{enumerate}
\usepackage{hyperref}

\def\RR{{\mathbb R}}
\def\ZZ{{\mathbb Z}}

\def\dist{{\rm dist}}
\def\C{{\mathcal C}}
\def\NN{{\mathbb N}}
\def\MM{{\mathbb M}}
\def\Vect{{\rm span}}
\def\Rect{({\mathsf{Rec}})}
\def\Trian{({\mathsf{Tri}})}
\def\Err{\mathcal{E}_h}
\def\Errm{\mathscr{M}_h}
\def\Rk{{\rm rk}}
\def\T{{\mathcal T}}
\def\A{{\mathscr N}}
\def\c{{\Lambda}}
\def\f{f}

\newlength{\AccoHaut}

\newtheorem{lemma}{Lemma}[section]%
\newtheorem{theorem}{Theorem}[section]%
\newtheorem{corollary}{Corollary}[section]%
\newtheorem{remark}{Remark}[section]%

\begin{document}
\title[On the error estimate of gradient inclusions]{On the error estimate of gradient inclusions}
\author[Omar Anza Hafsa]{Omar Anza Hafsa}

\address{Laboratoire LMGC, UMR-CNRS 5508, Place Eug\`ene Bataillon, 34095 Montpellier, France.}
\address{Universit\'e de N\^\i mes, Laboratoire MIPA, Site des Carmes, Place Gabriel P\'eri, 30021 N\^\i mes, France}

\email{omar.anza-hafsa@univ-montp2.fr}

\begin{abstract} 
The numerical analysis of gradient inclusions in a compact subset of $2\times 2$ diagonal matrices is studied. Assuming that the boundary conditions are reached after a finite number of laminations and using piecewise linear finite elements, we give a general error estimate in terms of the number of laminations and the mesh size. This is achieved by reduction results from compact to finite case. 
\end{abstract}
\maketitle

\section{Introduction and main result}\label{intro}
We denote by $\{\T_h\subset\RR^2:h>0\}$ a family of regular triangulation and by $\Omega\subset\RR^2$ a bounded polygonal domain satisfying 
\[
\overline{\Omega}=\bigcup_{T\in \T_h}T
\]
where $h$ is the mesh size defined by $h:=\max\{h_T:T\in\tau_h\}$ with $h_T$ is the diameter of a triangle $T\in\T_h$. For each $h>0$, we set the finite element space
\[
V_0^h:=\left\{u\in\C(\overline{\Omega};\RR^2): u\lfloor_{T}\mbox{ is affine for all }T\in\T_h\mbox{ and } u=0 \mbox{ on }\partial\Omega\right\}
\]
where $\partial\Omega$ is the boundary of $\Omega$. Let $K\subset \MM^{2\times 2}$ be a compact set. The error analysis of gradient inclusions consists to provide an optimal estimation in terms of the mesh size $h$ of 
\begin{align*}
\Errm^K:=\inf\left\{\Err^K(u):u\in V_0^h\right\}\quad\mbox{ with }\quad V_0^h\ni u\mapsto\Err^K(u):=\vert\{x\in\Omega:\nabla u(x)\notin K \}\vert,
\end{align*}
where $\vert \cdot\vert$ denotes the Lebesgue measure on $\RR^2$. 

We can assume that $0\notin K$, otherwise take $u=0$ then $\Err^K(u)=\Errm^K=0$. The construction of such $u\in V_0^h$ has to take account of the lack of rank-one compatibility of matrices of $K$. Two matrices $A,B\in\MM^{2\times  2}$ are said to be rank-one compatible if and only if $\mbox{\rm rk}(A-B)\le 1$.  This concept is related to the Hadamard lemma (see for instance \cite{Chipot2}) which states roughly that the gradients of a continuous affine function in two contiguous regions $\Omega_1$ and $\Omega_2$ of $\Omega$, i.e. $\nabla u=A$ on $\Omega_1$ and $\nabla u=B$ on $\Omega_2$, have to satisfy $\mbox{\rm rk}(A-B)\le 1$. 

Error analysis in terms of the mesh size $h$ has been accomplished in \cite{Chipot1, oah-chipot05} (see also \cite{bartels-prohl, bo-li}) for particular situations where $K$ was composed of four elements of diagonal $2\times 2$ matrices non rank-one compatible. In these situations, the boundary data $0$ was reached after $L$ laminations with $L=1,2$ in \cite{Chipot1} and $L=3,4$ in \cite{oah-chipot05}. This can be formulated by writing $0\in K^{(L)}$ where $L$ is the {\em level of laminations} defined as the integer $L=\inf\left\{i\in\NN: 0\in K^{(i)}\right\}$ (with the convention $\inf\emptyset=+\infty$), where for each $i\in\NN^\ast$
\[
\begin{cases}
K^{(0)}=K\\
K^{(i)}=\left\{\lambda F+(1-\lambda)G: F,G\in K^{(i-1)},\lambda\in [0,1],\mbox{ and
}\Rk(F-G)\le 1\right\}.
\end{cases}
\]
The set $K^{(i)}$ is called the {\em $i$-th lamination convex hull}. Our goal is to generalize these particular situations by assuming that {$L$ is finite and $K$ is a compact subset of the set of $2\times 2$ diagonal matrices $\MM^{2\times 2}_{\rm d}$}. For any $C>0$ we set
\[
V_{0,C}^h:=\left\{u\in V_0^h:\Vert\nabla u\Vert_{L^\infty(\Omega;\RR^2)}\le C\right\}.
\]
\begin{theorem}\label{main result} Assume that $1\le L<+\infty$, and $0\in K^{(L)}$ with $K\subset \MM^{2\times 2}_{\rm d}$ a compact set. Then there exists a finite set $\Sigma\subset K$ such that 
\[
\mbox{\rm card}(\Sigma)\le 2^{L+1} \quad \mbox{ and }\quad 0\in \Sigma^{(L)}, 
\] 
and there exist $C_L>0$ and $h_L>0$ such that for every $h\in ]0,h_L[$ there exists $\widehat{u}\in V_{0,C_L}^h$
\[
\Errm^K\le \Err^\Sigma(\widehat{u})\le C_Lh^{1\over 1+L}.
\]
\end{theorem}
Although we can find a similar result in \cite{bo-li}, their estimate is obtained (in a different framework) with respect to a particular finite element decomposition of $\Omega$, more precisely the triangulation is chosen depending on $K$. In our work, we obtain estimate of the same order $\mathcal{O}(h^{1\over 1+L})$ but {\em it is valid for any regular family of triangulation and for any compact subset of $\MM^{2\times 2}_{\rm d}$}. 
The following result is an easy consequence of Theorem\ref{main result}.                                                                                        
\begin{corollary}\label{main corollary} Let $\f:\MM^{2\times 2}\to [0,+\infty[$ be a Borel measurable function bounded on bounded sets such that 
$
K:=\left\{\xi\in\MM^{2\times 2}: \f(\xi)=0\right\}\subset\MM^{2\times 2}_{\rm d}\mbox{ is compact }.
$
Assume that $1\le L<+\infty$ and $0\in K^{(L)}$. Then there exist $C_L>0$ and $h_L>0$ such that for every $h\in ]0,h_L[$ we have
\[
\inf\left\{\int_\Omega \f(\nabla u(x))dx: u\in V_0^h\right\}\le C_Lh^{1\over 1+L}.
\]  
\end{corollary}
\subsection*{Outline of the paper} Section 2 is devoted to some preliminaries concerning geometry of diagonal matrices and approximations lemma. Section 3 is concerned with the proof of Theorem\ref{main result} which is divided in two parts. In the first part, we show how to deal with one and two laminations, this allows us to highlight the role of geometry of rank one compatibility of matrices of $K$. In the second part, we consider the general case of $L$ laminations whose analysis can be carried out by using reduction results from compact to finite case (see Proposition~\ref{choice} and Proposition~\ref{prop0}).
\section{Preliminaries}
\subsection{Rank-one compatibility in ${\MM^{2\times 2}_{\rm d}}$} 
In the following of the paper, it will be useful to consider the identification
\[
\MM^{2\times 2}_{\rm d}\ni\begin{pmatrix}x_{11}&0\\ \\0&x_{22}\end{pmatrix}\mapsto \begin{pmatrix}x_{11}\\ \\ x_{22}\end{pmatrix}\in\RR^2.
\] 

We state some facts about rank-one compatibility. Consider $\{A,B\}\in \MM^{2\times 2}_{\rm d}$ satisfying $\Rk(A-B)\le
1$, then $A-B=\alpha E_l$ for some $\alpha\in\RR$ and some
$l\in\{1,2\}$, with 
\[
E_1=\begin{pmatrix}1&0\\ \\0&0\end{pmatrix}\quad\mbox{ and }\quad E_2=\begin{pmatrix}0&0\\ \\0&1\end{pmatrix}.
\] 
Another formulation is to write $A\in B+\Vect(E_l)=A+\Vect(E_l)$ where $\Vect(E_l):=\{\alpha E_l:\alpha\in\RR\}$ is the one-dimensional vector space spanned by $E_l$. Thus, if we denote the interval $[A,B]:=\{\lambda A+(1-\lambda)B: \lambda\in [0,1]\}$ then $[A,B]\subset A+\Vect(E_l)$. 

\begin{lemma}\label{greatest-interval} Let $A\in \MM^{2\times 2}_{\rm d}$. Let $S\subset \MM^{2\times 2}_{\rm d}$ be a compact set. Let $l\in \{1,2\}$ be such that $A+\Vect(E_l)\cap S\not=\emptyset$. Then there exist $F,G\in A+\Vect(E_l)\cap S$ such that 
\begin{align}\label{greatest-interval-eq}
A+\Vect(E_l)\cap S\subset [F,G].
\end{align}
\end{lemma}
\begin{proof} The line $A+\Vect(E_l)$ is naturally isomorphic to $\RR$ with the isomorphism $\eta:A+\Vect(E_l)\to \RR$ is defined by $\eta(A+xE_l)=x$ with $x\in\RR$. We have that $ A+\Vect(E_l)\cap S$ is a compact subset of $A+\Vect(E_l)$. It follows that $\eta\left(A+\Vect(E_l)\cap S\right)$ is a compact subset of $\RR$, so there are an upper and a lower bounds $f,g\in \eta\left(A+\Vect(E_l)\cap S\right)$. Take $F:=\eta^{-1}\left(f\right)$ and $G:=\eta^{-1}\left(g\right)$, then \eqref{greatest-interval-eq} holds. 
\end{proof}

\begin{remark}\label{greatest-interval-remark} The Lemma~\ref{greatest-interval} will be used as follows: if $\widetilde{F}, \widetilde{G}\in S$ and satisfy $Rk(\widetilde{F}-\widetilde{G})=1$ then we can consider the greatest interval $[F,G]$ containing $[\widetilde{F},\widetilde{G}]$ with $\{F,G\}\subset S$. Indeed, since $\widetilde{F}, \widetilde{G}\in \widetilde{F}+\Vect(E_l)\cap S$ for some $l\in\{1,2\}$, we can apply Lemma~\ref{greatest-interval} to find $F,G\in \widetilde{F}+\Vect(E_l)\cap S$ such that $\widetilde{F}+\Vect(E_l)\cap S\subset [F,G]$. Thus $[\widetilde{F},\widetilde{G}]\subset [F,G]$.

\end{remark}
\subsection{Observations on the $i$-th lamination convex hull} The following result is used in Proposition~\ref{choice}. The proof follows easily by induction.
\begin{lemma}\label{compactness} For every $i\in\NN$, the set $K^{(i)}$ is compact.
\end{lemma}
The following lemma shows that the level of lamination $L$ can be as large as we wish (see Fig.~\ref{fig0}).
\begin{lemma} For every $n\in\NN^{*}$ there exists a finite set $K\subset\MM^{2\times 2}_{\rm d}$ such that
\[
{\rm card}(K)=n+1\;\mbox{ and }\;0\in K^{(n)}\setminus K^{(n-1)}.
\]
\end{lemma}
\begin{proof} Assume that $n=1$. Set $K_1=\{W_0^1,W_1^1\}$ with
\[
W_0^1=\begin{pmatrix}0&0\\ \\0&-1\end{pmatrix}\quad\mbox{ and }\quad W_1^1=\begin{pmatrix}0&0\\ \\0&1\end{pmatrix}.
\]
Then $0={1\over 2}W_0^1+{1\over 2}W_1^1$, it follows that $0\in K_1^{(1)}=[W_0^1,W_1^1]$ and $0\notin K_1^{(0)}=K$. 

Assume that $n=2$. We consider $[W_1^2,W_2^2]\ni W_1^1$ where
\[
W_2^2=\begin{pmatrix}-1&0\\ \\0&1\end{pmatrix}\quad\mbox{ and }\quad W_1^2=\begin{pmatrix}1&0\\ \\0&1\end{pmatrix}.
\]
We set $W_0^2=W_0^1$ and $K_2=\{W_0^2,W_1^2,W_2^2\}$. Then $0\in K_2^{(2)}\setminus K_2^{(1)}$. 

The construction of $K_n$ for $n\ge 2$ follows by induction. For each $n\in\NN^\ast$, we denote $W_i^n:=\begin{pmatrix}a_i^n&0\\ 0&b_i^n\end{pmatrix}$ where $i\in\{0,\dots,n\}$, and $a_i^n,b_i^n\in\RR$.

Assume that for $n=s$ with $s\ge 2$, we found $K_s=\{W_i^s\in\MM^{2\times 2}_{\rm d}:i\in \{0,\dots,s\}\}$, such that $0\in K_s^{(s)}\setminus K_s^{(s-1)}$. Then define $K_{s+1}=\{W_i^{s+1}\in\MM^{2\times 2}_{\rm d}:i\in \{0,\dots,s+1\}\}$ as follows:
\begin{enumerate}[label=(\roman*)]
\item $W_i^{s+1}=W_i^s$ for all $i\in\{0,\dots,s-1\}$;\\

\item $W_s^{s+1}=\begin{pmatrix}a_s^{s+1}&0\\ \\ 0&b_s^{s+1}\end{pmatrix}$ and $W_{s+1}^{s+1}=\begin{pmatrix}a_{s+1}^{s+1}&0\\ \\ 0&b_{s+1}^{s+1}\end{pmatrix}$ satisfy:\\

\begin{itemize}
\item if $s$ is odd, then we choose $a_s^s<a_s^{s+1}<a_{s-3}^s$, $b_s^{s+1}=b_s^s$, and $a_{s+1}^{s+1}<a_s^s$, $b_{s+1}^{s+1}=b_s^s$; it follows that $K_{s+1}^{(1)}=[W_s^{s+1},W_{s+1}^{s+1}]$, $K_{s+1}^{(s)}=K_{s+1}^{(1)}\cup K_s^{(s-1)}$ and $K_{s+1}^{(s+1)}=K_{s+1}^{(s)}\cup [W_0^{s+1},W_1^1]$. Thus we have $0\in K_{s+1}^{(s+1)}\setminus K_{s+1}^{(s)}$;\\

\item if $s$ is even, then we choose $a_s^{s+1}=a_s^s$, $b_{s-2}^{s+1}<b_s^{s+1}<b_s^s$ and $a_{s+1}^{s+1}=a_s^{s+1}$, $b_{s+1}^{s+1}>b_{s}^{s+1}$; as above, the same conclusion can be drawn.
\end{itemize}
\end{enumerate}
\end{proof}
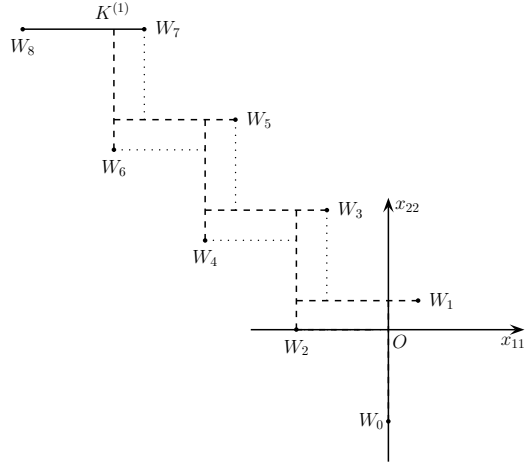
\begin{figure}[H]
\scalebox{0.5} 
{\fontsize{0.5cm}{0cm}\selectfont
\begin{pspicture}(0,-6.15)(15.04,6.17)
\psline[linewidth=0.04cm,arrowsize=0.2cm 2.0,arrowlength=1.4,arrowinset=0.4]{<-}(10,1)(10,-6)
\psline[linewidth=0.04cm,arrowsize=0.2cm 2.0,arrowlength=1.4,arrowinset=0.4]{<-}(13.58,-2.5)(6.38,-2.5)
\psdots[dotsize=0.12](10,-4.93)
\psdots[dotsize=0.12](10.78,-1.73)
\psdots[dotsize=0.12](7.58,-2.5)
\psdots[dotsize=0.12](8.38,0.67)
\psdots[dotsize=0.12](5.18,-0.13)
\psdots[dotsize=0.12](5.98,3.07)
\psdots[dotsize=0.12](2.78,2.27)
\psdots[dotsize=0.12](3.58,5.47)
\psdots[dotsize=0.12](0.38,5.47)
\psline[linewidth=0.04cm](0.38,5.47)(3.58,5.47)
\psline[linewidth=0.04cm,linestyle=dashed,dash=0.16cm 0.16cm](2.78,5.47)(2.78,2.27)
\psline[linewidth=0.04cm,linestyle=dashed,dash=0.16cm 0.16cm](2.78,3.07)(5.98,3.07)
\psline[linewidth=0.04cm,linestyle=dashed,dash=0.16cm 0.16cm](7.58,-2.5)(9.55,-2.5)
\psline[linewidth=0.04cm,linestyle=dashed,dash=0.16cm 0.16cm](5.18,3.07)(5.18,-0.13)
\psline[linewidth=0.04cm,linestyle=dashed,dash=0.16cm 0.16cm](5.18,0.67)(8.38,0.67)
\psline[linewidth=0.04cm,linestyle=dashed,dash=0.16cm 0.16cm](7.58,0.67)(7.58,-2.5)
\psline[linewidth=0.04cm,linestyle=dashed,dash=0.16cm 0.16cm](7.58,-1.73)(10.78,-1.73)
\psline[linewidth=0.04cm,linestyle=dotted,dotsep=0.16cm](3.58,5.47)(3.58,3.07)
\psline[linewidth=0.04cm,linestyle=dotted,dotsep=0.16cm](2.78,2.27)(5.18,2.27)
\psline[linewidth=0.04cm,linestyle=dotted,dotsep=0.16cm](5.98,3.07)(5.98,0.67)
\psline[linewidth=0.04cm,linestyle=dotted,dotsep=0.16cm](5.18,-0.13)(7.58,-0.13)
\psline[linewidth=0.04cm,linestyle=dotted,dotsep=0.16cm](8.38,0.67)(8.38,-1.73)
\psline[linewidth=0.04cm,linestyle=dotted,dotsep=0.16cm](7.58,-2.5)(10,-2.5)
\psline[linewidth=0.04cm,linestyle=dashed,dash=0.16cm 0.16cm](10,-1.73)(10,-4.93)
\usefont{T1}{ptm}{m}{n}
\rput(0.38,5.0){$W_{8}$}
\usefont{T1}{ptm}{m}{n}
\rput(2.78,1.8){$W_{6}$}
\usefont{T1}{ptm}{m}{n}
\rput(5.18,-0.6){$W_{4}$}
\usefont{T1}{ptm}{m}{n}
\rput(7.58,-2.97){$W_{2}$}
\usefont{T1}{ptm}{m}{n}
\rput(9.55,-4.93){$W_{0}$}
\usefont{T1}{ptm}{m}{n}
\rput(11.4,-1.73){$W_{1}$}
\usefont{T1}{ptm}{m}{n}
\rput(9,0.67){$W_{3}$}
\usefont{T1}{ptm}{m}{n}
\rput(6.6,3.07){$W_{5}$}
\usefont{T1}{ptm}{m}{n}
\rput(4.2,5.47){${W_{7}}$}
\usefont{T1}{ptm}{m}{n}
\rput(2.74,5.975){$K^{(1)}$}
\usefont{T1}{ptm}{m}{n}
\rput(10.3,-2.825){$O$}
\usefont{T1}{ptm}{m}{n}
\rput(13.28,-2.825){$x_{11}$}
\usefont{T1}{ptm}{m}{n}
\rput(10.5,0.775){$x_{22}$}
\end{pspicture} 
}
{\small\caption{Higher level of laminations: $L=8$.}\label{fig0}}
\end{figure}
\subsection{Approximation and local estimate result}
Let $Q=(a,b)\times (c,d)$ be a rectangle in $\RR^2$, with
$a,b,c,d\in\RR_{+}$. Let $\{W,F,G\}\subset\MM^{2\times 2}_{\rm d}$ be
defined by
\[
W=\begin{pmatrix}\alpha_0&0\\0&\beta_0\end{pmatrix},\quad
F=\begin{pmatrix}\alpha_1&0\\0&\beta_1\end{pmatrix}\mbox{ and }G=\begin{pmatrix}\alpha_2&0\\0&\beta_2\end{pmatrix},
\]
where $\alpha_i,\beta_i\in\RR$ with $i\in\{0,1,2\}$. Let $v_1,v_2\in
\C(\overline{Q};\RR)$ be such
that\footnote{It means that $v_1$ (resp. $v_2$) does not depend on $x_2$
  (resp. $x_1$).}
$v_i(x)=v_i(x_i)$ for all $i=1,2$. Consider $v=(v_1,v_2):\overline{Q}\to\RR^2$ be such that
\begin{eqnarray*}
\nabla v(\cdot)=W=\begin{pmatrix}\alpha_0&0\\0&\beta_0\end{pmatrix}
\;\mbox{ on }\;Q,\mbox{ with }(\alpha_0,\beta_0)\in\RR^2.
\end{eqnarray*}
Let $\delta_1=b-a$ and $\delta_2=d-c$. Set 
\[
\partial_1 Q:=\{(x_1,x_2)\in\partial Q: x_1=a \mbox{ or
}x_1=b\}
\] 
and 
\[
\partial_2 Q:=\{(x_1,x_2)\in\partial Q: x_2=c \mbox{ or
}x_2=d\}.
\] 
Let $k\in\NN^*$. Let $\tau,\sigma\in\RR$ such that $\tau-\sigma>0$, and $\mu\in ]0,1[$, we associate the set
\[
Q_i(\tau-\sigma,\mu,k)=\left\{x\in Q:\dist(x,\partial_j Q)>(\tau-\sigma)\mu\frac{\delta_i}{k}\right\}
\]
where $i,j\in\{1,2\}$ satisfy $ij=2$. 

Let $e\in\{a,c\}$, $p\in\{0,\dots,k-1\}$, $i\in\{1,2\}$, and $\mu_1,\mu_2\in [0,1]$, we set
\[
I_i(e,p,\mu_1,\mu_2,k):=\left]e+(p+\mu_1){\delta_i\over k},e+(p+\mu_2){\delta_i\over k}\right[.
\]
The following result is well-known, it allows us first to build Lipschitz functions with prescribed rank-one compatible gradients, and to give the error estimate for the set of ``bad gradients". It is used in the proof of Theorem~\ref{main result} for the construction of $\widehat{u}$.
\begin{lemma}\label{prop1} Assume that for some $\lambda\in ]0,1[$
\[
W=\lambda F+(1-\lambda)G\mbox{ and }\mbox{\rm rk}(F-G)=1.
\] 
Let $k\in\NN^\ast$. There exists a Lipschitz function $w:\overline{Q}\to \RR^2$ satisfying 
\begin{enumerate}[label=(\roman*)]
\item\label{i-local-approx}
$
w=v \mbox{ on }\partial Q\quad\mbox{ and }\quad
\Vert\nabla w\Vert_\infty\le 2(1+\Vert\nabla v\Vert_\infty);\\
$

\item\label{ii-local-approx}  if $F-G\in \Vect(E_1)$ then
\[
\begin{array}{ll}
\nabla w(x)=A & \mbox{ on strips }\;\displaystyle\bigcup_{p=0}^{{k}-1}I_1(a,p,0,\mu,k)\times ]c,d[\;\cap\;Q_\mu; \\ 
\nabla w(x)=B & \mbox{ on strips }\;\displaystyle\bigcup_{p=0}^{{k}-1}I_1(a,p,\mu,1,k)\times ]c,d[\;\cap\;Q_\mu,
\end{array}
\]
where
\begin{enumerate}[label=(\alph*)]
\item if $\alpha_1-\alpha_0>0$ then $A=F,\;B=G,\; \mu=\lambda, Q_\mu=Q_1(\alpha_1-\alpha_0,\mu,k)$ and 
\[
\big\vert\big\{x\in Q:\nabla
  w(x)\notin\{F,G\}\big\}\big\vert\le \vert F-G\vert\lambda{(b-a)^2\over k};
\]
\item if $\alpha_2-\alpha_0>0$ then $A=G,\; B=F,\; \mu=1-\lambda, Q_\mu=Q_1(\alpha_2-\alpha_0,\mu,k)$ and
\[
\big\vert\big\{x\in Q:\nabla
  w(x)\notin\{F,G\}\big\}\big\vert\le \vert F-G\vert(1-\lambda){(b-a)^2\over k};
\]
\end{enumerate}
\item\label{iii-local-approx} if $F-G\in \Vect(E_2)$ then
\[
\begin{array}{ll}
\nabla w(x)=A & \mbox{ on strips }\;\displaystyle\bigcup_{p=0}^{{k}-1}]a,b[\times I_2(c,p,0,\mu,k)\;\cap\;Q_\mu; \\ 
\nabla w(x)=B & \mbox{ on strips }\;\displaystyle\bigcup_{p=0}^{{k}-1}]a,b[\times I_2(c,p,\mu,1,k)\;\cap\;Q_\mu,
\end{array}
\]
where
\begin{enumerate}[label=(\alph*)]
\item if $\beta_1-\beta_0>0$ then $A=F,\;B=G,\; \mu=\lambda, Q_\mu=Q_2(\beta_1-\beta_0,\mu,k)$ and 
\[
\big\vert\big\{x\in Q:\nabla
  w(x)\notin\{F,G\}\big\}\big\vert\le \vert F-G\vert\lambda{(d-c)^2\over k};
\]
\item if $\beta_2-\beta_0>0$ then $A=G,\; B=F,\; \mu=1-\lambda, Q_\mu=Q_1(\beta_2-\beta_0,\mu,k)$ and
\[
\big\vert\big\{x\in Q:\nabla
  w(x)\notin\{F,G\}\big\}\big\vert\le \vert F-G\vert(1-\lambda){(d-c)^2\over k}.
\]
\end{enumerate}

\end{enumerate}
\end{lemma}
\begin{proof} Let $\mu\in ]0,1[$ and let $\chi_\mu$ be the periodic function of period $1$
  defined on $]0,1[$ by
\begin{align}\label{chi-mu}
\chi_\mu(\tau)=\left\{
\begin{array}{ll} 1 & \mbox{ if
    }\tau\in ]0,\mu]   \\ 
 0 & \mbox{ if
    }\tau\in ]\mu,1[.   
\end{array}
\right.
\end{align}
Note that we have either $\alpha_1-\alpha_0>0$ or $\alpha_2-\alpha_0>0$ since $\alpha_1-\alpha_0=(1-\lambda)(\alpha_1-\alpha_2)$ and $\alpha_2-\alpha_0=\lambda(\alpha_2-\alpha_1)$. 
\begin{enumerate}[label=(\alph*)]
\item If $\alpha_1-\alpha_0>0$, then we define $\omega_1:\overline{Q}\to \RR$ by
\[
\omega_1(x)=\omega_1(x_1)=v_1(a)+\int_{(a,x_1)}\left((\alpha_1-\alpha_2)\chi_{\lambda}\left({k\over\delta_1}{(t-a)}\right)+\alpha_2\right)dt.
\]
\item If $\alpha_2-\alpha_0>0$, then we define $\omega_1:\overline{Q}\to \RR$ by
\[
\omega_1(x)=\omega_1(x_1)=v_1(a)+\int_{(a,x_1)}\left((\alpha_2-\alpha_1)\chi_{1-\lambda}\left({k\over\delta_1}{(t-a)}\right)+\alpha_1\right)dt.
\]
\end{enumerate}
Let
$\hat{v}_1:\overline{Q}\to \RR$ be defined by
\[
\hat{v}_1(x):=\min\left\{\omega_1(x),v_1(x)+\dist(x,\partial_2 Q)\right\}.
\]
It is easy to verify that the function $w:=(\hat{v}_1,v_2)$ satisfies~\ref{i-local-approx} and\ref{ii-local-approx}. Similar constructions lead to~\ref{iii-local-approx}. 
\end{proof}

\begin{remark} Consider $W\in K^{(i-1)}$ for some $i\in\NN^*$ and some compact subset $K\subset\MM^{2\times 2}_{\rm d}$.  Let $F,G\in K^{(i)}$ such that $W\in [F,G]$ and let us call ${\rm diam}(K^{(i)})$ the diameter of $K^{(i)}$.  Then under the assumptions of the Lemma~\ref{prop1} and by setting $v^{(i-1)}=v$, we have that for every $k\in\NN^*$ there exists $v^{(i)}\in v^{(i-1)}+W^{1,\infty}_0(Q;\RR^2)$ satisfying
\[
\left\vert\left\{x\in Q:\nabla v^{(i)}(x)\notin K^{(i)}\right\}\right\vert\le {\rm diam}(K^{(i)}){{\rm diam}(Q)^2\over k}.
\]
\end{remark}

\bigskip

\bigskip

 We will divide the proof of Theorem~\ref{main result} into two sections. The first section is concerned with the case $L=1,2$ and the second one with $L\ge 2$. Throughout the proof, we denote by $C$ a positive constant which does not depend on the mesh size $h>0$. Let us denote by $\vert\partial\Omega\vert$ the length of the boundary of $\Omega$, $\delta={\rm diam}(\Omega)$ the diameter of $\Omega$ and $\widehat{\Omega}=\Omega\cap\left(\cup\{{\rm int}(T):T\in\T_h\}\right)$ for any $h>0$.
\section{Proof of Theorem\ref{main result}: $L\in\{1,2\}$}
\subsection{One lamination: $L=1$}\label{L=1}

\subsubsection*{Geometry} Since $0\in K^{(1)}$ there exists $\{W_1,W_2\}\subset K$ such that
\[
0\in [W_1,W_2] \;\mbox{ and }\;\Rk(W_1-W_2)\le 1.
\]
Obviously $\Rk(W_1-W_2)=1$ since $L=1$. Let $\lambda\in ]0,1[$ be such that $0=\lambda W_1+(1-\lambda)W_2$. Without loss of generality, we can assume that 
\[
W_1=\begin{pmatrix}a_1&0\\0&0\end{pmatrix}\;\mbox{ and }\;W_2=\begin{pmatrix}a_2&0\\0&0\end{pmatrix},
\]
with $a_1>0$ and $a_2={\lambda\over \lambda-1}a_1$. Set $\Sigma=\{W_1, W_2\}\subset K$, we have ${\rm card}(\Sigma)=2\le 2^{2}$ and $0\in \Sigma^{(1)}$.

\bigskip

Now, the proof consists in the construction of an element $\widehat{u}\in V_0^h$ such that $\nabla \widehat{u}(\cdot)\in \Sigma$  except a small set which will be evaluated in terms of the mesh size $h$.

\subsubsection*{Construction} Let $\alpha\in ]0,1[$ and $h>0$ such that 
\begin{equation}\label{hdelta}
h<1\hbox{ and }h^\alpha\le \delta.
\end{equation}
Let $v^{(1)}:[0,h^\alpha]\times \RR\to \RR^2$ be defined
by $v^{(1)}(x):=(v_1(x),0)$ where 
\[
v_1(x_1,x_2)=\left\{
\begin{array}{ll} a_1x_1 & \mbox{ if
    }(x_1,x_2)\in [0,\lambda h^\alpha] \times \RR  \\ \\
a_2(x_1-h^\alpha)  & \mbox{ if
    }(x_1,x_2)\in ]\lambda h^\alpha,h^\alpha] \times \RR.   
\end{array}
\right.\]
It is clear that $\nabla v^{(1)}(\cdot)\in \Sigma=\{W_1,W_2\}$ in $]0,h^\alpha[\times \RR$. Since for every $x_2\in\RR$ we have $v^{(1)}(0,x_2)=v^{(1)}(h^\alpha,x_2)=0$, we extend by periodicity in the direction $x_1$ with period $h^\alpha$, we obtain a function still denoted $v$. Consider the restriction $w=({v_1}{\lfloor_\Omega},0)$. Now, in order to match the boundary conditions we set 
\[\widetilde{w}(\cdot):=\left(\min\left\{{v_1}{\lfloor_\Omega}(\cdot),\dist(\cdot,\partial\Omega)\right\},0\right).\]
Finally we consider the interpolant $u^\alpha$ of $\widetilde{w}$ on $\T_h$. 

\medskip

\subsubsection*{Error estimation} Let us evaluate $\Err^\Sigma(u^\alpha)=\lvert\{x\in\widehat{\Omega}: \nabla u^\alpha(x)\notin
  \Sigma\}\rvert$ in terms of $h$. 

Let $D_p:=\left(\{ph^\alpha\}\times \RR\right)\cap\Omega$ where $p\in\ZZ$. Note that the number $N_h={\rm card}\{p\in\ZZ: D_p\ne\emptyset\}$ of lines $x_1=ph^\alpha$ lying in $\Omega$ satisfies
\begin{equation}\label{number L1}
(N_h-1)h^\alpha\le \delta.
\end{equation}
By \eqref{hdelta} and \eqref{number L1} we obtain
\begin{equation}\label{number}
h\le h^\alpha\;\mbox{ and } \; N_h\le 2\delta h^{-\alpha}.
\end{equation}
It is sufficient to estimate the measure of the set $\{x\in \widehat{\Omega}: w(x)\not=u^\alpha(x)\}$. Indeed we have that $\{x\in \widehat{\Omega}: w(x)=u^\alpha(x)\}\subset \{x\in\widehat{\Omega}: \nabla u^\alpha(x)\in \Sigma\}$. We have $\{x\in \widehat{\Omega}: w(x)\not=u^\alpha(x)\}\subset V_1\cup V_2$ where $V_1$ and $V_2$ are given as follows.
\begin{enumerate}[label=(\roman*)]
\item $V_1$ is a neighborhood (in the direction $x_2$) of the lines $x_1=ph^\alpha$ in $\Omega$ of width $2h$  where possibly $w\not=u^\alpha$, i.e.,
\[ 
V_1:=\left\{x\in\Omega:\dist(x,D_p)<2h \mbox{ for some } p\in\ZZ\right\};
\]
its measure satisfies
\begin{equation}\label{onelamV1}
\vert V_1\vert\le \delta 2h N_h.
\end{equation}
\item $V_2=V_2^1\cup V_2^2$ where
 \[
V_2^1:=\left\{x\in\Omega: \dist(x,\partial\Omega)<\Vert w_1\Vert_{L^\infty(\Omega;\RR^2)}\right\}\mbox{ and } V_2^2:=\left\{x\in\Omega:\dist(x,\Gamma_h)<2h\right\}
\]
with $\Gamma_h=\{x\in\Omega:\dist(x,\partial\Omega)=\Vert w_1\Vert_{L^\infty(\Omega;\RR^2)}\}$ and $\Vert w_1\Vert_{L^\infty(\Omega;\RR^2)}=\lambda a_1h^\alpha$;
\begin{itemize} 
\item $V_2^1$ is the error introduced by the boundary condition whose measure satisfies
\[
\left\vert V_2^1\right\vert\le \vert\partial\Omega\vert \lambda a_1 h^{\alpha};
\] 
\item $V_2^2$ is the error introduced by the interpolation near $\partial\Omega$ whose measure satisfies
\[
\left\vert V_2^2\right\vert\le \vert\partial\Omega\vert 2h.
\] 
\end{itemize}
Therefore 
\begin{equation}\label{onelamV2}
\vert V_2\vert\le\vert\partial\Omega\vert (\lambda a_1h^\alpha+2h).
\end{equation}
\end{enumerate}
Collecting \eqref{onelamV1} and \eqref{onelamV2}, we obtain  
\[
\Err^\Sigma(u^\alpha)\le \vert V_1\vert+\vert V_2\vert\le\vert\partial\Omega\vert (\lambda a_1 h^{\alpha}+2h)+\delta 2h N_h.
\]
Using (\ref{number}), we find for every $\alpha\in ]0,1[$
\[
\Err^\Sigma(u^\alpha)\le C\left(h^\alpha+h^{1-\alpha}\right),
\]
where $C=\max\big(\vert\partial\Omega\vert\max\{\sup_{\xi\in \Sigma}\vert\xi\vert,2\},4\delta^2\big)$. Now, it is easy to see that the function $\alpha\mapsto h^\alpha+h^{1-\alpha}$ is minimum for $\widehat\alpha:={1\over 2}$. Set $\widehat{u}:=u^{\widehat \alpha}$. It follows that for some $C>0$
\[
\Errm^K\le\Errm^\Sigma\le\Err^\Sigma(\widehat u)\le C h^{1\over 2}.
\]
Note that since $\Vert\nabla \widetilde{w}\Vert_{L^\infty(\Omega;\RR^2)}\le \Vert\nabla v^{(1)}\Vert_{L^\infty(\Omega;\RR^2)}\le \sup_{\xi\in \Sigma}\vert\xi\vert$ and the mesh is regular, we deduce that $\widehat{u}\in V_{0,C}^h$ for some $C>0$ independent of $h$. The proof is complete for the case $L=1$.
\subsection{Two laminations: $L=2$}\label{L=2} 
\subsubsection*{Geometry} In this case $0\in K^{(2)}$. There exists $\{\widetilde{F}_1^1,\widetilde{F}_1^2\}\subset K^{(1)}$ such that 
\[
0\in [\widetilde{F}_1^1,\widetilde{F}_1^2]\;\mbox{ and }\;\Rk(\widetilde{F}_1^1-\widetilde{F}_1^2)\le 1.
\]
Obviously $\Rk(\widetilde{F}_1^1-\widetilde{F}_1^2)= 1$ since $L=2$. There exists $l\in\{1,2\}$ such that $[\widetilde{F}_1^1,\widetilde{F}_1^2]\subset \widetilde{F}_1^1+\Vect(E_l)$. By Remark~\ref{greatest-interval-remark}, we consider the greatest interval $[F_1^1,F_1^2]$ containing $[\widetilde{F}_1^1,\widetilde{F}_1^2]$ and such that $\{F_1^1,F_1^2\}\subset K^{(1)}$. One of the two possibilities can occur: either $\{F_1^1,F_1^2\}\subset K^{(1)}\setminus K$, or $\{F_1^1,F_1^2\}\cap K$ contains one element only. So we divide the proof into two cases.


\paragraph*{\em Case 1.} We assume that 
\begin{equation}\label{case1_L_2}
\{F_1^1,F_1^2\}\subset K^{(1)}\setminus K\;\mbox{ and }\;F_1^1-F_1^2\in\Vect(E_l)\;\mbox{ for some }l\in\{1,2\}.
\end{equation} 
There exists $\{W_1,W_2\}\subset K$ such that $F_1^1\in [W_1,W_2]$ and $\Rk(W_1-W_2)\le 1$. Since $F_1^1\notin K$, we deduce that $\Rk(W_1-W_2)=1$. Assume that $W_1-W_2\in \Vect(E_l)$, since $[F_1^1,F_1^2]$ is maximal we have $[W_1,W_2]\subset [F_1^1,F_1^2]$. It follows that $F_1^1\in\{W_1,W_2\}$ which contradicts (\ref{case1_L_2}). Thus $W_1-W_2\in\Vect(E_{3-l})$. The same conclusion can be drawn for $F_1^2$ with $\{W_3,W_4\}\subset K$ satisfying $F_1^2\in [W_3,W_4]$ and $\Rk(W_3-W_4)=1$. 

Now, consider the projection of $W_1$ and resp. $W_2$ parallel to $\Vect(E_l)$ on $W_3+\Vect(E_{3-l})$, which we denote $G_3$ and resp. $G_4$. Either $[G_3,G_4]\subset [W_3,W_4]$ or $\{G_3,G_4\}\cap [W_3,W_4]$ contains one element only (if not consider the projection of $W_3,W_4$ in place of $W_1,W_2$). Assume that $[G_3,G_4]\subset [W_3,W_4]$. Set $K^\prime=\{W_1,W_2,G_3,G_4\}$, note that $\{G_3,G_4\}\subset K^{(1)}$. We can choose $\{W_1^\prime,W_2^\prime,W_3^\prime,W_4^\prime\}\subset K^\prime$ in order to have
\begin{gather*}
W_1^\prime=\begin{pmatrix}a_1&0\\0&b_1\end{pmatrix}\quad  
W_2^\prime=\begin{pmatrix}a_1&0\\0&b_2\end{pmatrix}\quad   
W_3^\prime=\begin{pmatrix}a_2&0\\0&b_2\end{pmatrix}\quad  
W_4^\prime=\begin{pmatrix}a_2&0\\0&b_1\end{pmatrix},
\end{gather*}
with $a_1,b_1\ge 0$, $a_2,b_2\le 0$, satisfying $a_1-a_2>0$ and
$b_1-b_2>0$.

Assume that $\{G^\prime\}=\{G_3,G_4\}\cap [W_3,W_4]$, and let $H^{\prime\prime}\in\{W_1,W_2\}$ whose projection is $G^\prime$. Consider the projection of $W_3$ and resp. $W_4$ parallel to $\Vect(E_l)$ on $W_1+\Vect(E_{3-l})$, which we denote $G_1$ and resp. $G_2$. Then $\{G_1,G_2\}\cap [W_1,W_2]$ contains one element denoted $H^\prime$ which corresponds to the projection of $G^{\prime\prime}\in\{W_3,W_4\}$. Note that $\{H^\prime,G^{\prime\prime}\}\subset K^{(1)}$. Set $K^\prime=\{H^{\prime\prime}, G^\prime,H^\prime,G^{\prime\prime}\}$. We can choose $\{W_1^\prime,W_2^\prime,W_3^\prime,W_4^\prime\}\subset K^\prime$ as above. 

Let $j\in\{1,2,3,4\}$. There exists $\{W_j^1,W_j^2\}\subset K$ such that $W_j^\prime\in [W_j^1,W_j^2]$ and $\Rk(W_j^1-W_j^2)\le 1$. We set 
$
\Sigma=\cup_{j=1}^4\{W_j^1,W_j^2\}\subset K,
$
satisfying ${\rm card}(\Sigma)=8=2^{1+L}$ and $0\in \Sigma^{(2)}$.


\paragraph*{\em Case 2.} Assume that $\{\widetilde{F}_1^1,\widetilde{F}_1^2\}\cap K$ contains one element only. Let $\lambda\in]0,1[$ be such that $0=\lambda \widetilde{F}_1^1+(1-\lambda)\widetilde{F}_1^2$. Without loss of generality, we assume that $\widetilde{F}_1^1=W_1\in K$, $W_1-\widetilde{F}_1^2\in\Vect(E_1)$, and
\begin{gather*}
W_1=\begin{pmatrix}a_1&0\\0&0\end{pmatrix}\;\mbox{ and }\;
\widetilde{F}_1^2 =\begin{pmatrix}a_2&0\\0&0\end{pmatrix}
\end{gather*}  
with $a_1>0$ and $a_2={\lambda\over \lambda-1}a_1$. Since $0\in K^{(2)}$, there exists $\{W_2,W_3\}\subset K$ such that $\widetilde{F}_1^2=\mu W_2+(1-\mu)W_3$ for some $\mu\in ]0,1[$, with 
\begin{gather*}
W_2=\begin{pmatrix}a_2&0\\0&b_2\end{pmatrix}\;\mbox{ and }\;
W_3 =\begin{pmatrix}a_2&0\\0&b_3\end{pmatrix}.
\end{gather*}  
Without loss of generality, we assume that $b_2-b_3>0$. We set $\Sigma=\{W_1,W_2,W_3\}$ which satisfies ${\rm card}(\Sigma)\le 2^{1+L}$ and $0\in \Sigma^{(2)}$.

\bigskip

Now, we deal with the construction of $\widehat{u}\in V_0^h$ such that $\nabla \widehat{u}(\cdot)\in \Sigma$ except a small set which will be evaluated in terms of the mesh size $h$. 

\subsubsection*{Construction: {\it Case $\it 1$}} Let $\alpha\in ]0,1[$ and $h>0$ such that 
\begin{equation}\label{hdelta2}
h<1\hbox{ and }h^\alpha\le \delta.
\end{equation}
Roughly, the strategy is to begin by building a function $v^{(1)}$ on the square $[0,2h^\alpha]^2$ such that $\nabla v^{(1)}(\cdot)\in K^\prime$ a.e.. Then, since each $W_j^\prime\in [W_j^1,W_j^2]$, we modify the function into $v^{(2)}$ in order to have $\nabla v^{(2)}(\cdot)\in \Sigma=\cup_{j=1}^4\{W_j^1,W_j^2\}\subset K$ except a small set of ``bad gradients". 

Let $v^{(1)}:[0,2h^\alpha]^2\to\RR^2$ be defined by $v^{(1)}:=(v_1,v_2)$ where
\[
v_1(x_1,x_2):=\left\{
\begin{array}{ll} a_1x_1 & \mbox{ if
    }(x_1,x_2)\in[0,2h^\alpha s_1]\times [0,2h^\alpha]  \\ a_2x_1-2h^\alpha a_2 & \mbox{ if
    }(x_1,x_2)\in[2h^\alpha s_1,2h^\alpha]\times [0,2h^\alpha]   
\end{array}
\right.\]

\[
v_2(x_1,x_2):=\left\{
\begin{array}{ll}
 b_1x_2 & \mbox{ if }(x_1,x_2)\in [0,2h^\alpha]\times [0,2h^\alpha s_2]   \\
  b_2x_2-2h^\alpha b_2 & \mbox{ if }(x_1,x_2)\in [0,2h^\alpha]\times[2h^\alpha s_2,2h^\alpha]
  
\end{array}
\right.
\]
with $s_1={a_2\over a_2-a_1}$ and $s_2={b_2\over b_2-b_1}$. We have 
\[
\nabla v^{(1)}(\cdot)\in K^\prime:=\big\{W_1^\prime,W_2^\prime,W_3^\prime,W_4^\prime\big\} \;\;\mbox{ in }]0,2h^\alpha[^2.
\]
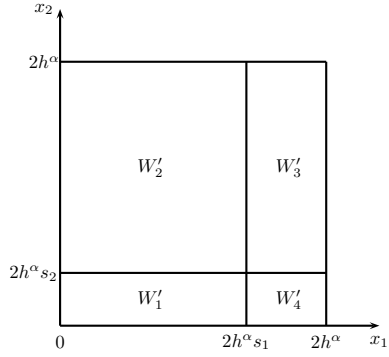
\begin{figure}[H]
\scalebox{0.7} 
{
\begin{pspicture}(0,-3.4)(7,3.4)
\psline[linewidth=0.04cm,arrowsize=0.05291667cm 2.0,arrowlength=1.4,arrowinset=0.4]{<-}(0,3)(0,-3)
\psline[linewidth=0.04cm,arrowsize=0.05291667cm 2.0,arrowlength=1.4,arrowinset=0.4]{<-}(6,-3)(0,-3)
\psline[linewidth=0.04cm](5,-3)(5,2)
\psline[linewidth=0.04cm](0,2)(5,2)
\psline[linewidth=0.04cm](0,-2)(5,-2)
\psline[linewidth=0.04cm](3.5,2)(3.5,-3)

\usefont{T1}{ptm}{m}{n}
\rput(6,-3.3){$x_1$}
\usefont{T1}{ptm}{m}{n}
\rput(-0.3,3){$x_2$}
\usefont{T1}{ptm}{m}{n}
\rput(0,-3.3){$0$}
\usefont{T1}{ptm}{m}{n}
\rput(1.7,-2.5){$W_1^\prime$}
\usefont{T1}{ptm}{m}{n}
\rput(4.3,-2.5){$W_4^\prime$}
\usefont{T1}{ptm}{m}{n}
\rput(1.7,0){$W_2^\prime$}
\usefont{T1}{ptm}{m}{n}
\rput(4.3,0){$W_3^\prime$}

\usefont{T1}{ptm}{m}{n}
\rput(5,-3.3){$2h^\alpha$}

\usefont{T1}{ptm}{m}{n}
\rput(3.5,-3.3){$2h^\alpha s_1$}

\usefont{T1}{ptm}{m}{n}
\rput(-0.3,2){$2h^\alpha$}

\usefont{T1}{ptm}{m}{n}
\rput(-0.5,-2){$2h^\alpha s_2$}

\end{pspicture} 
}
{\small{\caption{Domain of $\nabla v^{(1)}$.}}\label{fig4}}
\end{figure} More precisely we have
\begin{eqnarray*}
&&\nabla v^{(1)}(\cdot)=W^\prime_1\quad\mbox{ in }\quad Q_1=]0,2h^\alpha s_1[\times ]0,2h^\alpha s_2[,\\
&&\nabla v^{(1)}(\cdot)=W^\prime_2\quad\mbox{ in }\quad Q_2=]0,2h^\alpha s_1[\times ]2h^\alpha s_2,2h^\alpha[,\\
&&\nabla v^{(1)}(\cdot)=W^\prime_3\quad\mbox{ in }\quad Q_3=]2h^\alpha s_1,2h^\alpha[\times ]2h^\alpha s_2,2h^\alpha[,\\
&&\nabla v^{(1)}(\cdot)=W^\prime_4\quad\mbox{ in }\quad Q_4=]2h^\alpha s_1,2h^\alpha[\times ]0,2h^\alpha s_2[.
\end{eqnarray*} 
Without loss of generality, we assume that $K^\prime\cap K=\emptyset$, if not we modify $v^{(1)}$ when it is necessary.

Let $k\in\NN^\ast$. Using Lemma~\ref{prop1} we replace $v^{(1)}$ by $w_k^l\in W^{1,\infty}(Q_l; \RR^2)$ on each $Q_l$ such that
\[
w_k^l=v^{(1)}\;\mbox{ on }\;\partial Q_l \;\mbox{ and }\;\left\Vert\nabla w_k^l\right\Vert_{L^\infty(Q_l;\RR^2)}\le 2\left(1+\left\Vert\nabla v^{(1)}\right\Vert_{L^\infty(Q_l;\RR^2)}\right)
\]  
for all $l\in\{1,2,3,4\}$. Define $v^{(2)}:[0,2h^\alpha]^2\to\RR^2$ by $v^{(2)}=w_k^l$ on $Q_l$ with $l\in\{1,2,3,4\}$. Extend $v^{(2)}$ by periodicity with period $2h^\alpha$ in both directions $x_1,x_2$, we obtain a function still denoted $v^{(2)}$. Consider $w_k=(w^1_k,w^2_k)={v^{(2)}}{\lfloor_\Omega}$. In order to match the boundary conditions we set \[\widetilde{w}_k(\cdot):=\left(\min\left\{w_k^1(\cdot),\dist(\cdot,\partial\Omega)\right\},\min\left\{w_k^2(\cdot),\dist(\cdot,\partial\Omega)\right\}\right).\] Finally, consider $u^\alpha_k$ the interpolant of $\widetilde{w}_k$ on $\T_h$. Since $\Vert\nabla \widetilde{w}_k\Vert_{L^\infty(\Omega;\RR^2)}\le C(1+\sup_{\xi\in \Sigma^{(2)}}\vert\xi\vert)$ for some $C>0$ which does not depend on $h$ (but possibly depends on $L$), and the mesh is regular, we deduce that $u^\alpha_k\in V_{0,C}^h$ for some $C>0$. 


\subsubsection*{Error estimation: {\it Case $\it 1$}} Let us evaluate $\Err^\Sigma(u^\alpha_k)=\vert\left\{x\in\Omega: \nabla u^\alpha_k(x)\notin\Sigma\right\}\vert$ in terms of $h$. 
  
Let $D_p^1:=\left(\{p2h^\alpha\}\times \RR\right)\cap\Omega\;\mbox{ and }\;D_p^2:=\left(\RR\times \{p2h^\alpha\}\right)\cap\Omega$ where $p\in\ZZ$. For each $i\in\{1,2\}$ the number $N_i={\rm card}\{p\in\ZZ: D_p^i\ne\emptyset\}$ of lines $x_i=ph^\alpha$ lying in $\Omega$ satisfies $(N_i-1)h^\alpha\le \delta$. By \eqref{hdelta2} it holds 
$
h\le h^\alpha\;\mbox{ and } \; N_i\le 2\delta h^{-\alpha}.
$
Therefore the number of squares $N_h$ of type $[0,2h^\alpha]^2$ lying in $\Omega$, i.e.,
\[
N_h:=\mbox{card}\left\{z\in\ZZ^2: z2h^\alpha+[0,2h^\alpha]^2\subset\Omega\right\}
\]
 satisfies 
\begin{equation}\label{number_squares}
N_h=(N_1-1)(N_2-1)\le 4\delta^2h^{-2\alpha}.
\end{equation}
Now, we have $\{x\in\widehat\Omega:\nabla u^\alpha_k(x)\notin\Sigma\}\subset V_1\cup V_2\cup V_3$ where $V_1,V_2$ and $V_3$ are given as follows. 
\begin{enumerate}[label=(\roman*)]
\item $\displaystyle
V_1:=\mathop{\cup}_{z\in\ZZ^2}\mathop{\cup}_{l\in\{1,\cdots,4\}}\left\{x\in z2h^\alpha+Q_l:\nabla w_k(x)\notin\Sigma\right\}\cap\widehat{\Omega}.
$
By Lemma~\ref{prop1} we find
\begin{equation}\label{S_3} 
\vert V_1\vert\le CN_h {h^{2\alpha}\over k}
\end{equation}
for some $C>0$. Note that $V_1$ is the set where possibly $\nabla u^\alpha_k(\cdot)\notin\Sigma$ because of the preservation of continuity of $w_k$ and thus of $u^\alpha_k$.
\item $V_2:=V_2^1\cup V_2^2$ where 
\[
V_2^1:=\mathop{\cup}_{z\in\ZZ^2} \left\{x\in z2h^\alpha+[0,2h^\alpha]^2: \dist(x,L_{s,z})<2h\right\}\cap\widehat{\Omega}
\] 
with 
\[
L_{s,z}:=\left\{(x_1,x_2)\in z2h^\alpha+[0,2h^\alpha]^2: x_i=z_i2h^\alpha\mbox{ or }x_i=z_i2h^\alpha s_i\mbox{ for some }i\in\{1,2\}\right\}
\] 
and 
\[
V_2^2:=\mathop{\cup}_{z\in\ZZ^2}\mathop{\cup}_{l\in\{1,\cdots,4\}} \left\{x\in z2h^\alpha+[0,2h^\alpha]^2: \dist(x,L_{s,z}^l)<2h\right\}\cap\widehat{\Omega}
\]
with \[L_{s,z}^l:=\mathop{\cup}_{p=0}^{k-1}\left\{(x_1,x_2)\in z2h^\alpha+Q_l: x_i=s^p_i\mbox{ or }x_i=t^p_i\mbox{ for some }i\in\{1,2\}\right\};\] 
\begin{itemize}
\item $V_2^1$ is the error introduced by the interpolation on lines $x_i=p2h^\alpha$ or $x_i=p2h^\alpha s_i$ of each squares of type $[0,2h^\alpha]^2$ where $i\in\{1,2\}$, its measure satisfies for some $C>0$
\[
\left\vert V_2^1\right\vert\le C N_h2h h^\alpha;
\]
\item $V_2^2$ is the error introduced by the interpolation on lines of each strip induced by the functions $w_k^l$, its measure satisfies for some $C>0$
\[
\left\vert V_2^2\right\vert\le C N_hk2h h^\alpha;
\]
\end{itemize}
it folllows that for some $C>0$
\begin{equation}\label{S_24} 
\left\vert V_2\right\vert\le CkN_hh^{1+\alpha}.
\end{equation}

\item\label{error-boundary-estimate} $V_3=V_3^1\cup V_3^2$ where
 \[
V_3^1:=\left\{x\in\Omega: \dist(x,\partial\Omega)<\Vert w_k\Vert_{L^\infty(\Omega;\RR^2)}\right\}\mbox{ and } V_3^2:=\left\{x\in\Omega:\dist(x,\Gamma_h)<2h\right\}
\]
with $\Gamma_h:=\{x\in\Omega:\dist(x,\partial\Omega)=\Vert w_k\Vert_{L^\infty(\Omega;\RR^2)}\}$ and $\Vert w_k\Vert_{L^\infty(\Omega;\RR^2)}=\max\{\Vert w_k^1\Vert_{L^\infty(\Omega)},\Vert w^2_k\Vert_{L^\infty(\Omega)}\}$; similarly to the estimation in the case of one lamination we find for some $C>0$
\begin{equation}\label{S_1}
\vert V_3\vert \le C h^\alpha.
\end{equation}

\end{enumerate}
Therefore collecting (\ref{S_3}), (\ref{S_24}) and (\ref{S_1}), we obtain by taking \eqref{number_squares} into account 
\[
\Err^\Sigma(u^\alpha_k) \le\sum_{i=1}^3\vert V_i\vert\le C\left(h^\alpha+h^{-2\alpha}\left({h^{2\alpha}\over k}+kh^{1+\alpha}\right)\right)=C\left(h^\alpha+{1\over k}+k{h^{1-\alpha}}\right).
\]
The function $x\mapsto \sigma(x):={1\over x}+x{h^{1-\alpha}}$ is minimum at ${x_1}=h^{-{1-\alpha\over 2}} $. If $x_1$ is not an integer then take $\widehat{k}:=\ceil [\big]{x_1}+1$, where $\ceil [\big]{\cdot}$ denotes the integer part. It is not difficult to see that for some $C>0$
\[
\inf_{k\ge 1}\left(h^\alpha+{1\over k}+k{h^{1-\alpha}}\right)\le C\left(h^\alpha+h^{{1-\alpha\over 2}}\right).
\]
Indeed, we have
\begin{align*}
\inf_{k\ge 1}\sigma(k)\le \sigma\left(\widehat{k}+1\right)\le h^\alpha+{h^{\frac{1-\alpha}{2}}}+{h^{\frac{1-\alpha}{2}}}+h^{1-\alpha}\le 3\left( h^\alpha+{h^{\frac{1-\alpha}{2}}}\right)
\end{align*}
since $h<1$ and $\alpha<1$.

The function $\alpha\mapsto h^\alpha+h^{{1-\alpha\over 2}}$ is minimum for $\widehat{\alpha}:={1\over 3}$. Set $\widehat{u}:=u^{\widehat\alpha}_{\widehat k}$. It follows that for some $C>0$
\[
\Errm^K\le \Errm^\Sigma\le\Err^\Sigma(\widehat u)\le C h^{1\over 3},
\]
which finishes the proof for {\it Case 1}.
\begin{remark} If we choose differents $k_l$ (in place of $k$) on each $Q_l$ with $l\in\{1,2,3,4\}$ for the number of ``laminations" then the estimate for $\Errm^K$ keeps unchanged. 
\end{remark}

\subsubsection*{Construction: Case $\it 2$} Let $\alpha\in ]0,1[$ and $h>0$ be such that 
\begin{equation}\label{hdelta21}
h<1\hbox{ and }h^\alpha\le \delta.
\end{equation}
Let $v^{(1)}:[0,h^\alpha]\times \RR\to \RR^2$ be defined
by $v^{(1)}(x):=(v_1(x),0)$, where 
\[
v_1(x)=\left\{
\begin{array}{ll} a_1x_1 & \mbox{ if
    }x_1\in [0,\lambda h^\alpha]   \\ \\
a_2(x_1-h^\alpha)  & \mbox{ if
    }x_1\in ]\lambda h^\alpha,h^\alpha].   
\end{array}
\right.\]
It is clear that $\nabla v^{(1)}(x)\in \{W_1,\widetilde{F}_1^2\}$ a.e. in $]0,h^\alpha[\times \RR$. 

Let $k\in\NN^\ast$. Consider $\varphi_k:[0,h^\alpha]\times \RR\to \RR$ defined by   
\[
\varphi_k(x)=\left\{
\begin{array}{ll} \nu_k(x_2)& \mbox{ if
    }(x_1,x_2)\in [\lambda h^\alpha,h^\alpha]\times \RR^+  \\ \\
-\nu_k(-x_2)& \mbox{ if
    }(x_1,x_2)\in [\lambda h^\alpha,h^\alpha]\times \RR^-\\ \\
0&\mbox{ if } (x_1,x_2)\in [0,\lambda h^\alpha]\times \RR  
\end{array}
\right.
\]
where 
\[
\RR^+\ni x_2\mapsto \nu_k (x_2)=\int_{0}^{x_2}\left((b_2-b_3)\chi_\mu\left({k\over h^\alpha}t\right)+b_3\right)dt,
\]
with $\chi_\mu$ is defined by \eqref{chi-mu}. 
Define $v_2:[0,h^\alpha]\times \RR \to\RR$ by 
\[
v_2(x):=\min\{\varphi_k(x),\dist(x,L_h)\}
\]
where $L_h=\{x\in [0,h^\alpha]\times \RR: x_1=\lambda h^\alpha\mbox{ and }x_1=h^\alpha\}$. Set $v^{(2)}:=(v_1,v_2)$. We extend it by periodicity in the direction $x_1$ of period $h^\alpha$, and we obtain a function still denoted $v^{(2)}$. Then consider the restriction to $\Omega$, i.e., $w_k:=(w^1_k,w^2_k)={v^{(2)}}{\lfloor_\Omega}$. In order to match the boundary conditions, we set
\[
\widetilde{w}_k(\cdot):=\left(\min\left\{w^1_k(\cdot),\dist(\cdot,\partial\Omega)\right\},\min\left\{w^2_k(\cdot),\dist(\cdot,\partial\Omega)\right\}\right).
\] 
Finally, consider the interpolant $u^\alpha_k$ of $\widetilde{w}_k$ on $\T_h$. Similarly to {\it Case 1}, $u^\alpha_k\in V_{0,C}^h$ for some $C>0$.

\subsubsection*{Error estimation: Case $\it 2$} It remains to evaluate $\Err^\Sigma(u^\alpha_k)$ in terms of $h$. We have $\{x\in\widehat{\Omega}:\nabla u^\alpha_k(x)\notin \Sigma\}\subset V_1\cup V_2\cup V_3$ where $V_1,V_2$ and $V_3$ are given as follows.
\begin{enumerate}[label=(\roman*)]
\item A neighborhood $V_1$ of $L_h$ defined by
\[
V_1:=\left\{x\in\Omega:\dist(x, zh^\alpha+L_h)<\Vert w_k\Vert_{L^\infty(\Omega;\RR^2)} \mbox{ for some }z\in\ZZ\right\}
\] where possibly $\nabla u^\alpha_k(\cdot)\notin \Sigma$. It satisfies for some $C>0$
\begin{equation}\label{s_4}
\vert V_1\vert\le C {h^{\alpha}\over k}h^{-\alpha}=C{1\over k}.
\end{equation}
\item An error is introduced by the interpolation which consists in a neighborhood $V_2$
\begin{itemize} 
\item of lines $x_1=ph^\alpha$ or $x_1=p\lambda h^\alpha$ lying in $\Omega$, with $p\in\ZZ$, its measure is bounded by
$C h^{1-\alpha}$ for some $C>0$;
\item of lines $x_2=ph^\alpha$ and $(p+\mu)h^\alpha$ lying in $\Omega$, with $p\in\ZZ$, its measure is bounded by
$C {h}kh^{-\alpha}=C{k}h^{1-\alpha}$ for some $C>0$;
\end{itemize}
it follows that for some $C>0$
\begin{equation}\label{V1case2}
\vert V_2\vert\le C{k}h^{1-\alpha}.
\end{equation}
\item Similarly to {\it Case 1}, $V_3$ is a neighborhood of $\partial\Omega$ satisfying for some $C>0$
\begin{equation}\label{s_1}
\left\vert V_3\right\vert\le C h^\alpha.
\end{equation}
\end{enumerate}
Collecting the bounds \eqref{s_4}, \eqref{V1case2} and \eqref{s_1}, we deduce for some $C>0$
\[
\Err^\Sigma(u^\alpha_k)\le\sum_{i=1}^3 \vert V_i\vert\le C\left(h^{\alpha}+{1\over k}+kh^{1-\alpha}\right),
\]
the end of the proof follows the one of {\it Case $\it1$}. 
\hfill$\blacksquare$
%
%
\section{Proof of Theorem~\ref{main result}: $L\ge 2$}
\subsection{Geometry: reduction from compact to finite case} We start by reduction results from compact to finite case. Roughly, we show that when  $L$ is finite then two configurations can occur. The first one is that there are four matrices in $K^{(L-1)}$ which are the vertices of a rectangle whose convex hull contains $0\in\MM^{2\times 2}_{\rm d}$ (see Figure~\ref{fig1}.). The second configuration consists in three matrices which are the vertices of a triangle whose convex envelope contains $0$, with one of its vertex belongs to $K$ and placed on one of the axis, and the two others vertices belong to $K^{(L-2)}$ (see Figure~\ref{fig2}.).

\begin{figure}[H]
\centering
   \begin{minipage}[l]{.46\linewidth}
     \scalebox{0.4} 
{\fontsize{0.5cm}{0cm}\selectfont
\begin{pspicture}(0,-5.8)(13,5.8)
\psdots[dotsize=0.14](4,3)
\psdots[dotsize=0.14](10,3)
\psdots[dotsize=0.14](4,-3)
\psdots[dotsize=0.14](10,-3)
\psline[linewidth=0.04cm,arrowsize=0.2cm 2.0,arrowlength=1.4,arrowinset=0.4]{<-}(11,0)(3,0)
\psline[linewidth=0.04cm,arrowsize=0.2cm 2.0,arrowlength=1.4,arrowinset=0.4]{<-}(7,4)(7,-4)
\psline[linewidth=0.04cm,linestyle=dashed,dash=0.16cm 0.16cm](4,3)(10,3)
\psline[linewidth=0.04cm,linestyle=dashed,dash=0.16cm 0.16cm](10,3)(10,-3)
\psline[linewidth=0.04cm,linestyle=dashed,dash=0.16cm 0.16cm](4,-3)(10,-3)
\psline[linewidth=0.04cm,linestyle=dashed,dash=0.16cm 0.16cm](4,3)(4,-3)
\rput{-50.0}(2.5,5.5){\psellipse[linewidth=0.04,linestyle=dotted,dotsep=0.06cm,dimen=outer](7.16,0)(5.3,4.8)}
\usefont{T1}{ptm}{m}{n}
\rput(5.8,4.4){{\bf\huge {\it {K}}}}
\usefont{T1}{ptm}{m}{n}
\rput(7.7,3.6){$x_{22}$}
\usefont{T1}{ptm}{m}{n}
\rput(10.7,-0.41){$x_{11}$}
\usefont{T1}{ptm}{m}{n}
\rput(10.7,3){$W^\prime_{1}$}
\usefont{T1}{ptm}{m}{n}
\rput(10.7,-3){$W^\prime_{2}$}
\usefont{T1}{ptm}{m}{n}
\rput(3.3,-3){$W^\prime_{3}$}
\usefont{T1}{ptm}{m}{n}
\rput(3.3,3){$W^\prime_{4}$}
\rput(7.45,-0.41){$O$}
\usefont{T1}{ptm}{m}{n}
\end{pspicture} 
}
\caption{\small{The~rectangle~case.}\label{fig1}}
   \end{minipage} 
   \begin{minipage}[r]{.46\linewidth}
    \scalebox{0.4} 
{\fontsize{0.5cm}{0cm}\selectfont
\begin{pspicture}(0,-5.8)(13,5.8)
\psdots[dotsize=0.14](4,3)
\psdots[dotsize=0.14](4,-3)
\psdots[dotsize=0.14](10,0)
\psdots[dotsize=0.14](4,0)
\psline[linewidth=0.04cm,arrowsize=0.2cm 2.0,arrowlength=1.4,arrowinset=0.4]{<-}(11,0)(3,0)
\psline[linewidth=0.04cm,arrowsize=0.2cm 2.0,arrowlength=1.4,arrowinset=0.4]{<-}(7,4)(7,-4)
\psline[linewidth=0.04cm,linestyle=dashed,dash=0.16cm 0.16cm](4,3)(4,-3)
\rput{-50.0}(2.5,5.5){\psellipse[linewidth=0.04,linestyle=dotted,dotsep=0.06cm,dimen=outer](7.16,0)(5.3,4.7)}
\usefont{T1}{ptm}{m}{n}
\rput(5.8,4.4){{\bf\huge {\it {K}}}}
\usefont{T1}{ptm}{m}{n}
\rput(10.2,0.4){$W_0$}
\usefont{T1}{ptm}{m}{n}
\rput(7.7,3.6){$x_{22}$}
\usefont{T1}{ptm}{m}{n}
\rput(3.3,3){$W^\prime_2$}
\usefont{T1}{ptm}{m}{n}
\rput(10.7,-0.41){$x_{11}$}
\usefont{T1}{ptm}{m}{n}
\rput(3.3,0.4){$W^\prime_1$}
\usefont{T1}{ptm}{m}{n}
\rput(3.3,-3){$W^\prime_{3}$}
\usefont{T1}{ptm}{m}{n}
\rput(7.45,-0.41){$O$}
\psline[linewidth=0.04cm,linestyle=dashed,dash=0.16cm 0.16cm](4,3)(10,0)
\psline[linewidth=0.04cm,linestyle=dashed,dash=0.16cm 0.16cm](4,-3)(10,0)
\end{pspicture} 
}
  \caption{\small{The~triangle~case.}\label{fig2}}
   \end{minipage}
\end{figure}
\begin{theorem}\label{choicecorollary} Assume that $2\le L<+\infty$. Then there exists $\Sigma\subset K$ such that 
\[
\mbox{\rm card}(\Sigma)\le 2^{L+1}\quad\mbox{ and }\quad 0\in \Sigma^{(L)},
\] 
and one of the two assertions holds:
\begin{enumerate}
\item[$\Rect$] there exists $\displaystyle K^\prime=\{W_1^\prime, W_2^\prime,W_3^\prime,W_4^\prime\}\subset\Sigma^{(L-1)}$ such that $0\in {\rm co}(K^\prime)$ the convex envelope of $K^\prime$ and 
\begin{gather*}
W_1^\prime=\begin{pmatrix}a_1&0\\0&b_1\end{pmatrix}\quad  
W_2^\prime=\begin{pmatrix}a_1&0\\0&b_2\end{pmatrix}\quad   
W_3^\prime=\begin{pmatrix}a_2&0\\0&b_2\end{pmatrix}\quad  
W_4^\prime=\begin{pmatrix}a_2&0\\0&b_1\end{pmatrix},
\end{gather*}
with $a_1,b_1\ge 0$, $a_2,b_2\le 0$, satisfying $a_1-a_2>0$ and $b_1-b_2>0$;\\

\item[$\Trian$] there exist $W_0\in K$, $W_1^\prime\in \Sigma^{(L-1)}\setminus \Sigma^{(L-2)}$, $W_2^\prime\in \Sigma^{(L-2)}$ and $W_3^\prime\in \Sigma^{(L-2)}$ such that for some $l\in\{1,2\}$
\begin{eqnarray*}
&& 0\in[W_1^\prime, W_0],\;\;\Rk(W_0-W_1^\prime)=1\;\mbox{ and }W_0-W_1^\prime\in\Vect(E_l);\\
&& W_1^\prime\in [W_2^\prime,W_3^\prime],\;\;\Rk(W_2^\prime-W_3^\prime)=1\;\mbox{ and }W_2-W_3^\prime\in\Vect(E_{3-l}).
\end{eqnarray*}
\end{enumerate}
\end{theorem}                         
\begin{proof}  
The proof is a consequence of Lemma~\ref{choice} and Lemma~\ref{prop0} below.
\end{proof}                   
\begin{lemma}\label{choice} Assume that $2\le L<+\infty$. Then one of the two assertions holds:
\begin{enumerate}
\item[$\Rect$] there
  exists $\displaystyle K^\prime=\{W_1^\prime, W_2^\prime,W_3^\prime,W_4^\prime\}\subset
  K^{(L-1)}$ such that $0\in {\rm co}(K^\prime)$ and 
\begin{gather*}
W_1^\prime=\begin{pmatrix}a_1&0\\0&b_1\end{pmatrix}\quad  
W_2^\prime=\begin{pmatrix}a_1&0\\0&b_2\end{pmatrix}\quad   
W_3^\prime=\begin{pmatrix}a_2&0\\0&b_2\end{pmatrix}\quad  
W_4^\prime=\begin{pmatrix}a_2&0\\0&b_1\end{pmatrix},
\end{gather*}
with $a_1,b_1\ge 0$, $a_2,b_2\le 0$, satisfying $a_1-a_2>0$ and
$b_1-b_2>0$;\\

\item[$\Trian$] there exist $W_0\in K$, $W_1^\prime\in K^{(L-1)}\setminus K^{(L-2)}$, $W_2^\prime\in K^{(L-2)}$ and $W_3^\prime\in K^{(L-2)}$ such that for some $l\in\{1,2\}$
\begin{eqnarray*}
&& 0\in[W_1^\prime, W_0],\;\;\Rk(W_0-W_1^\prime)=1\;\mbox{ and }W_0-W_1^\prime\in\Vect(E_l);\\
&& W_1^\prime\in [W_2^\prime,W_3^\prime],\;\;\Rk(W_2^\prime-W_3^\prime)=1\;\mbox{ and }W_2-W_3^\prime\in\Vect(E_{3-l}).
\end{eqnarray*}
\end{enumerate}
\end{lemma}
\begin{proof} Assume that $2\le L<+\infty$. Then there exists $\{\widetilde{F}_1^1,\widetilde{F}_1^2\}\subset K^{(L-1)}\setminus K^{(L-2)}$ such that
\[
0\in [\widetilde{F}_1^1,\widetilde{F}_1^2] \;\mbox{ and }\;\Rk(\widetilde{F}_1^1-\widetilde{F}_1^2)= 1. 
\]
By Lemma~\ref{compactness} together with Remark~\ref{greatest-interval-remark} we consider the greatest interval $[F_1^1,F_1^2]$ containing
$[\widetilde{F}_1^1,\widetilde{F}_1^2]$ such that $\{F_1^1,F_1^2\}\subset
K^{(L-1)}$. It is clear that
\[
0\in [F_1^1,F_1^2] \;\mbox{ and }\;\Rk(F_1^1-F_1^2)= 1. 
\]
We claim that one of the two possibilities holds:
\begin{enumerate}
\item[$\Rect^\prime$]\label{C1} there exists $i\in\{1,\dots,L-1\}$ such that either $F_1^1\in K^{(i)}\setminus K^{(i-1)}$ and $F_1^2\in K^{(L-1)}\setminus K^{(L-2)}$, or $F_1^2\in K^{(i)}\setminus K^{(i-1)}$ and $F_1^1\in K^{(L-1)}\setminus K^{(L-2)}$;\\

\item[$\Trian^\prime$]\label{C2} either $F_1^1\in K$ and $F_1^2\in K^{(L-1)}\setminus K^{(L-2)}$, or $F_1^2\in K$ and $F_1^1\in K^{(L-1)}\setminus K^{(L-2)}$.
\end{enumerate}

\bigskip

Indeed, we have one of the two assertions:
\begin{equation}\label{condition2}
\{F_1^1,F_1^2\}\not\subset K^{(L-1)}\setminus K^{(L-2)};
\end{equation}
\begin{equation}\label{condition1}
\{F_1^1,F_1^2\}\subset K^{(L-1)}\setminus K^{(L-2)}.
\end{equation}
Assume that (\ref{condition2}) holds. Note that we cannot have $\{F_1^1,F_1^2\}\subset K^{(L-2)}$ since $L$ is minimal. Thus, we have either $F_1^1\in K^{(L-1)}\setminus K^{(L-2)}$ and $F_1^2\in K^{(L-2)}$, or $F_1^2\in K^{(L-1)}\setminus K^{(L-2)}$ and $F_1^1\in K^{(L-2)}$. Without loss of generality, assume that $F_1^1\in K^{(L-1)}\setminus K^{(L-2)}$ and $F_1^2\in K^{(L-2)}$. Since 
\[
K^{(L-2)}=\bigcup_{i=0}^{L-2} K^{(i)}\;\mbox{ and }\; K^{(i)}\subset K^{(i+1)}\;\mbox{ for all }i\in\{0,\dots, L-1\},
\]
it follows that either there exists $i\in\{1,\dots,L-2\}$ such that $F_1^2\in K^{(i)}\setminus K^{(i-1)}$, which is, on account of (\ref{condition2}), the claim $\Rect^\prime$. Or for every $i\in\{1,\dots,L-2\}$, we have $F_1^2\notin K^{(i)}\setminus K^{(i-1)}$, i.e., $F_1^2\in K$, which is the claim $\Trian^\prime$.

\medskip

{\em To prove $\Rect$}. Assume that $\Rect^\prime$ holds. We have $[F_1^1,F_1^2]\subset
F_1^1+\Vect(E_l)$ for some $l\in\{1,2\}$. Without loss of generality, assume that there exists $i\in\{1,\dots,L-1\}$ such that $F_1^1\in K^{(i)}\setminus K^{(i-1)}$ and $F_1^2\in K^{(L-1)}\setminus K^{(L-2)}$. There exists $\{G_2^{1},G_2^{2}\}\subset K^{(i-1)}$ such that
$F_1^1\in [G_2^{1},G_2^{2}]$ and $\Rk(G_2^{1}-G_2^{2})\le
1$. We must have $\Rk(G_2^{1}-G_2^{2})=1$, otherwise we would have
$G_2^{1}=G_2^{2}=F_1^j\notin K^{(i-1)}$, a contradiction. Assume
that $G_2^{1}\in G_2^{2}+\Vect(E_l)$. Hence there exists
$A\in\{G_2^{1},G_2^{2}\}$ such that $[A,F_1^{2}]\supset
[F_1^1,F_1^2]$. By the fact that $[F_1^1,F_1^2]$ is maximal together with $A\in
K^{(i-1)}\subset K^{(L-1)}$, we obtain $A=F_1^1\notin K^{(i-1)}$, which is
impossible. Therefore $G_2^{2}\in G_2^{1}+\Vect(E_{3-l})$. The same reasoning applies to $F_1^2$, and we obtain $\{G_2^{3},G_2^{4}\}\subset K^{(L-2)}$ such that
$F_1^2\in [G_2^{3},G_2^{4}]$, $\Rk(G_2^{3}-G_2^{4})=1$ and $G_2^{4}\in G_2^{3}+\Vect(E_{3-l})$.

Thus
$\Delta_1=G_2^1+\Vect(E_{3-l})$ is parallel to
$\Delta_2=G_2^3+\Vect(E_{3-l})$. Consider the orthogonal projection to
$\Vect(E_{3-l})$ of $G_2^1,G_2^2$ on $\Delta_2$ (resp. of
$G_2^3,G_2^4$ on $\Delta_1$) which we call
$\widetilde{G}_2^3,\widetilde{G}_2^4$
(resp. $\widetilde{G}_2^1,\widetilde{G}_2^2$).

Assume that 
\[
[\widetilde{G}_2^3,\widetilde{G}_2^4]\subset [G_2^3,G_2^4] \;\left(\mbox{or }\;[\widetilde{G}_2^1,\widetilde{G}_2^2]\subset[G_2^1,G_2^2]\right).
\]
Then, define
\[
K^\prime=\left\{G_2^1,G_2^2,\widetilde{G}_2^3,\widetilde{G}_2^4\right\}\left(\mbox{or }\;K^\prime=\left\{\widetilde{G}_2^1,\widetilde{G}_2^2,G_2^3,G_2^4\right\}\right).
\]
Note that, for instance, if $[\widetilde{G}_2^3,\widetilde{G}_2^4]\subset
[G_2^3,G_2^4]$ then $\widetilde{G}_2^3,\widetilde{G}_2^4\in K^{(L-1)}$. The proof follows by a suitable choice of $W_i^\prime$ (with
$i=1,\dots,4$) in $K^\prime$ satisfying the conclusion of the theorem.

Otherwise, there exists $A\in [G_2^3,G_2^4]$ (resp. $B\in
[G_2^1,G_2^2]$) such that $A$ (resp. $B$) is the projection of $G_2^1$
or $G_2^2$ (resp. of $G_2^3$ or $G_2^4$). Without loss of generality,
assume that $A$ (resp. $B$) is the projection of $G_2^1$ (resp. of
$G_2^3$). Then, Define 
\[
K^\prime=\left\{G_2^1, A, G_2^3, B\right\}.
\]
As above, the proof follows by a suitable choice of $W_i^\prime$ (with
$i=1,\dots,4$) in $K^\prime$ satisfying the conclusion of the theorem.

\medskip

{\em To prove $\Trian$}. Assume that $\Trian^\prime$ holds. Without loss of generality, we can assume that $F_1^1\in K$ and $F_1^2\in K^{(L-1)}\setminus K^{(L-2)}$. Set $W_0:=F_1^2$ and $W_1^\prime:=F_1^2$, of course we have $W_0\not=W_1^\prime$ since $L$ is minimal. There exists $\{A,B\}\subset K^{(L-2)}$ such that $W_1^\prime\in [A,B]$ and $\Rk(A-B)\le 1$. Obviously, we have $\Rk(A-B)=1$, if not $W_1^\prime\in K^{(L-2)}$ which is impossible. By Lemma \ref{compactness} together with Remark~\ref{greatest-interval-remark}, we consider the greatest interval $[W_2^\prime,W_3^\prime]$ containing $[A,B]$ such that $\{W_2^\prime,W_3^\prime\}\subset K^{(L-2)}$. It holds that $W_1^\prime\in\Vect(E_l)$ for some $l\in\{1,2\}$. Suppose that $W_2^\prime\in W_3^\prime+\Vect(E_l)$, since $[W_0,W_1^\prime]$ is maximal we deduce that $W_1^\prime\in K^{(L-2)}$ which is impossible. Thus $W_2^\prime\in W_3^\prime+\Vect(E_{3-l})$. The proof is complete.
\end{proof}

                                       
The following result allows us, starting from an element of $K^{(L-1)}$, to choose successive decompositions which are in orthogonal directions.
\begin{lemma}\label{prop0} Assume that $3\le
L<+\infty$. Let $i\in\{1,...,L-2\}$ and $j\in\{1,...,2^i\}$. Let
  $F_i^j\in K^{(L-i)}$. Assume that there exist
  $\{F_{i+1}^{2j-1},F^{2j}_{i+1}\}\subset K^{(L-i-1)}$ and $l\in\{1,2\}$ such that
$$
F_i^j\in [F_{i+1}^{2j-1},F_{i+1}^{2j}]\mbox{ and }F_{i+1}^{2j-1}-F_{i+1}^{2j}\in \Vect(E_l).
$$
Then we have four possibilities:
\begin{enumerate}[label=(\arabic*)]
\item\label{c1} there exists
$\left\{F_{i+2}^{4j-3},F^{4j-2}_{i+2}\right\}\subset
K^{(L-i-2)}$ such that 
\begin{eqnarray*}
F_{i+1}^{2j-1}\in [F_{i+2}^{4j-3},F^{4j-2}_{i+2}],\;F_{i+2}^{4j-3}-F^{4j-2}_{i+2}\in\Vect(E_{3-l})\;\mbox{ and }\Rk(F_{i+2}^{4j-3}-F^{4j-2}_{i+2})=1;
\end{eqnarray*}
\item\label{c2} there exists
$\left\{F^{4j-1}_{i+2},F^{4j}_{i+2}\right\}\subset
K^{(L-i-2)}$ such that 
\begin{eqnarray*}
F_{i+1}^{2j}\in [F_{i+2}^{4j-1},F^{4j}_{i+2}],\;F_{i+2}^{4j-1}-F^{4j}_{i+2}\in\Vect(E_{3-l})\;\mbox{ and }\;\Rk(F_{i+2}^{4j-1}-F^{4j}_{i+2})=1;
\end{eqnarray*} 
\item\label{c3} $F_{i+1}^{2j-1}\in K$;
\item\label{c4} $F_{i+1}^{2j}\in K$.
\end{enumerate}
\end{lemma}
\begin{proof} Let $i\in\{1,...,L-2\}$ and $j\in\{1,...,2^i\}$. Let
  $F_i^j\in K^{(L-i)}$. Assume that there exist
  $\{\widetilde{F}_{i+1}^{2j-1},\widetilde{F}^{2j}_{i+1}\}\subset K^{(L-i-1)}$ and $l\in\{1,2\}$ such that
\[
F_i^j\in [\widetilde{F}_{i+1}^{2j-1},\widetilde{F}_{i+1}^{2j}]\mbox{ and }\widetilde{F}_{i+1}^{2j-1}-\widetilde{F}_{i+1}^{2j}\in \Vect(E_l).
\]
Consider the greatest
interval $[F_{i+1}^{2j-1},F_{i+1}^{2j}]$ containing
$[\widetilde{F}_{i+1}^{2j-1},\widetilde{F}_{i+1}^{2j}]$ and such that
$\{F_{i+1}^{2j-1},F_{i+1}^{2j}\}\subset K^{(L-i-1)}$. We have one of the two possibilities:
\begin{enumerate}[label=(\alph*)]
\item\label{p2} for every $r\in\{i+1,\dots,L-1\}$ if $F_{i+1}^{2j}\in
  K^{(L-r)}$ then $F_{i+1}^{2j}\in
  K^{(L-r-1)}\subset K^{(L-i-2)}$;
\item\label{p1} there exists $r\in\{i+1,\dots,L-1\}$ such that $F_{i+1}^{2j}\in
  K^{(L-r)}$ and $F_{i+1}^{2j}\notin
  K^{(L-r-1)}$.
\end{enumerate}
If \ref{p2} holds then $F_{i+1}^{2j}\in K$ and we obtain \ref{c4}. Otherwise, assume \ref{p1}
holds. There exists $\{F^{4j-1}_{i+2},F^{4j}_{i+2}\}\subset
K^{(L-r-1)}\subset K^{(L-i-2)}$ such that 
\begin{eqnarray*}
F_{i+1}^{2j}\in
[F_{i+2}^{4j-1},F^{4j}_{i+2}]\;\mbox{
  and }\;\Rk(F_{i+2}^{4j-1}-F^{4j}_{i+2})\le 1.
\end{eqnarray*} 
If $\Rk(F_{i+2}^{4j-1}-F^{4j}_{i+2})=0$ then $F_{i+1}^{2j}\in K^{(L-r-1)}$, which is impossible. Therefore $\Rk(F_{i+2}^{4j-1}-F^{4j}_{i+2})=1$. If
$F^{4j-1}_{i+2}-F^{4j}_{i+2}\in \Vect(E_l)$ then $[F^{4j-1}_{i+2},F^{4j}_{i+2}]\subset [F_{i+1}^{2j-1},F_{i+1}^{2j}]$. But $F_{i+1}^{2j}\in [F_{i+2}^{4j-1},F^{4j}_{i+2}]$ which is possible only if $F_{i+1}^{2j}\in K^{(L-r-1)}$, a contradiction. Hence $F^{4j-1}_{i+2}-F^{4j}_{i+2}\in \Vect(E_{3-l})$, and we obtain \ref{c2}. Similar arguments apply to $F_{i+1}^{2j-1}$ to obtain \ref{c3} and \ref{c1}. The proof is complete.
\end{proof}

In the following, we show how to construct $\widehat{u}\in V_0^h$ and give estimate of $\mathcal{E}^\Sigma_h(\widehat{u})$ in the case of the configuration $\Rect$. The same estimate can be achieved for the case $\Trian$ by taking into account of the differences in the construction of $\widehat{u}\in V_0^h$ as in {\it Case 2} of Subsection \ref{L=2}.  
\subsection{Construction in the case $\Rect$}\label{h1-subsection}
We summarize the construction of $\widehat u$ as follows.

\bigskip

\noindent{(1)} Using Theorem \ref{choicecorollary}, we start by constructing a function $v^{(1)}$ defined on the square $[0,2h^{\alpha}]^2$ with $\alpha\in ]0,1[$, such that $\nabla v^{(1)}(\cdot)\in K^\prime\subset \Sigma^{(L-1)}$ and $v^{(1)}=0$ on the boundary of the square. More presicely, let $\alpha\in (0,1)$ and $h\in ]0,h_1[$ with $h_1:=\min\{1,\delta^{1\over \alpha}\}$. Let $v^{(1)}:[0,2h^\alpha]^2\to\RR^2$ be defined by $v^{(1)}:=(v_1,v_2)$ where
\[
v_1(x_1,x_2)=\left\{
\begin{array}{ll} a_1x_1 & \mbox{ if
    }x_1\in[0,2h^\alpha s_1]   \\ a_2x_1-2h^\alpha a_2 & \mbox{ if
    }x_1\in[2h^\alpha s_1,2h^\alpha]   
\end{array}
\right.\]

\[
v_2(x_1,x_2)=\left\{
\begin{array}{ll}
 b_1x_2 & \mbox{ if }x_2\in[0,2h^\alpha s_2]   \\
  b_2x_2-2h^\alpha b_2 & \mbox{ if }x_2\in[2h^\alpha s_2,2h^\alpha]
  
\end{array}
\right.
\]
with $s_1={a_2\over a_2-a_1}$ and $s_2={b_2\over b_2-b_1}$. We have 
\[
\nabla v^{(1)}(x)\in K^\prime=\big\{W_1^\prime,W_2^\prime,W_3^\prime,W_4^\prime\big\} \;\;\mbox{ a.e. in }[0,2h^\alpha]^2.
\]
Then it remains $L-2$ laminations to reach $\Sigma\subset K$.

\bigskip

\noindent{(2)} For each $W_s^\prime\in K^\prime$ with $s\in\{1,\cdots,4\}$, there exists $\{F_{2,s}^1,F_{2,s}^2\}\in \Sigma^{(L-2)}$ such that $\Rk(F_{2,s}^1-F_{2,s}^2)\le 1$ and $W_s^\prime\in[F_{2,s}^1,F_{2,s}^2]$. We can apply Lemma~\ref{prop1} to find a function, which we call $v^{(2)}$, such that $\nabla v^{(2)}(\cdot)\in \{F_{2,s}^1,F_{2,s}^2\}$ except a small set whose measure is provided by Lemma ~\ref{prop1}; on account of Lemma~\ref{prop0}, we can decompose each $F_i^j\in K^{(L-i)}$ in orthogonal direction with respect to the direction of the previous decomposition. So, the divisions in strips are successively orthogonal and we build successive functions $v^{(3)},\dots, v^{(L)}$ with respect to these subdivisions. (According to that specific construction, the estimate of  the measure of the set $\nabla v^{(L)}\notin K$ will be easier to evaluate). 

More presicely, we let $s\in\{1,\dots,4\}$ and we assume that the number of laminations needed to reach $K$ when starting with $W^\prime_s$ is $L_s\le L$. We have 
\[
W_s^\prime=\lambda_{1,s}^1F_{2,s}^1+(1-\lambda_{1,s}^1)F_{2,s}^2\quad\mbox{ with
}\quad\lambda_{1,s}^1\in [0,1]\quad\mbox{ and }\quad F_{2,s}^1,F_{2,s}^2\in K^{(L-2)}.
\]
For every $i\in\{3,\dots,L_s\}$ and every $j\in\{1,\dots,2^i\}$
\begin{align*}
F_{i,s}^j=\lambda_{i,s}^{j}F_{i,s}^{2j}+(1-\lambda_{i,s}^j)F_{i,s}^{2j-1}\;\mbox{ with
}\quad\lambda_{i,s}^j\in [0,1]\quad\mbox{ and }\quad F_{i,s}^{2j},F_{i,s}^{2j-1}\in K^{(L-i)}.
\end{align*}

Let $M=\sup\{\vert\xi\vert:\xi\in \Sigma^{(L)}\}$. Let
$\{k_{i}\}_{i=1,\dots,L}\subset\NN^\ast$ with $k_{1}=1$, be such
that $k_{i}>\c k_{i-1}$ for all $i\in\{2,\dots,L\}$ with
\[
\c=M\left(\min_{1\le s\le 4}\min_{1\le i\le L}\min_{1\le j\le 2^i}\left\{{\lambda_{i,s}^j},{1-\lambda_{i,s}^j}\right\}\right)^{-1}.
\]
Set $\A_1:=\NN^*$ and for each $i\in\{2,\dots,L\}$ 
\[
\A_i:=\{k\in\NN:k>\c k_{i-1}\}.
\]
The bound $\c$ allows us to make the successive subdivisions increasingly thin. Using Lemma~\ref{prop0} and Lemma~\ref{prop1}, we construct functions $v_{k_2}:=v^{(2)},v_{k_2,k_3}:=v^{(3)},\dots,v_{k_2,\dots,k_L}:=v^{(L)}$ corresponding to $L-2$ laminations.

\bigskip

\noindent{(3)} We extend by periodicity of period $2h^\alpha$ the function $v_{k_2,\dots,k_L}$ in both directions $x_1,x_2$. In order to match the zero boundary conditions, we set 
\[
\widetilde{w}_{k_2,\dots,k_{L}}(\cdot):=\left(\min\left\{v^1_{k_2,\dots,k_L}(\cdot),\dist(\cdot,\partial\Omega)\right\},\min\left\{v^2_{k_2,\dots,k_L}(\cdot),\dist(\cdot,\partial\Omega)\right\}\right),
\]
and finally we define $u^\alpha_{k_2,\dots,k_{L}}$ the interpolant of $\widetilde{w}_{k_2,\dots,k_{L}}$ on $\T_h$.
                                     
\subsection{Error estimation in the case $\Rect$} Let $j\in\NN^\ast$. We denote by 
\[
I_j:=
\begin{cases}
\{0,1,2,3,\dots,{j-1\over 2}\}&\mbox{ if }j \mbox{ is odd}\\
\{0,1,2,3,\dots,{j\over 2}\}&\mbox{ if }j \mbox{ is even}
\end{cases}
\]
thus
\[
\prod_{l\in I_j}k_{j-2l}=\begin{cases}
k_jk_{j-2}\cdots k_3k_1&\mbox{ if }j \mbox{ is odd}\\
k_jk_{j-2}\cdots k_4k_2&\mbox{ if }j \mbox{ is even}.
\end{cases}
\]

\medskip

After $i-1$ successive modifications of the function $v^{(1)}$, we obtain a function $v^{(i)}$, where $i\in\{2,\dots,L\}$. We estimate the set of ``bad gradients", i.e., where the function $\nabla v^{(i)}$ possibly does not belong to $K$. 

Let a cell be such that its length is of order 
\[
\prod_{l\in I_{i-2\over 2}}{h^{\alpha}\over k_{i-2-2l}}
\] 
and its width is of order 
\[
\prod_{l\in I_{i-1\over 2}}{h^{\alpha}\over k_{i-1-2l}}.
\]
We assume that the $(i-1)$-th subdivision is made in the direction $x_1$. By Lemma~\ref{prop0} together with Lemma~\ref{prop1}, we construct a function $v^{(i)}$ and the $i$-th subdivision is made in the direction $x_2$. Thus, the cell is divided in $k_i$ strips of width of order 
\[
\prod_{l\in  I_{i\over 2}}{h^{\alpha}\over k_{i-2l}}.
\] 
\subsubsection*{Error estimation due to the preservation of continuity} By Lemma~\ref{prop1} the estimation is bounded by, for some $C>0$
\[
C\left(\left({h^\alpha\over \displaystyle\prod_{l\in I_{i-2\over 2}}k_{i-2-2l}}\right)^2{1\over k_i}\right)\times N_s
\]
where $N_s$ is the number of these cells in each rectangle $Q_s$ which is bounded by
\[
C k_{i-1}k_{i-2}\cdots k_2k_1=\prod_{l=1}^{i-1}k_l.
\]
We deduce that for some $C>0$ 
\[
\left\vert \left\{x\in Q_s:\nabla v_{k_2,\dots,k_i}(x)\notin \Sigma\right\} \right\vert\le Ch^{2\alpha}{\displaystyle \prod_{l=0}^{1+{i\over 2}}k_{i-2l-1}\over \displaystyle \prod_{l=0}^{i\over 2}k_{i-2l}}=Ch^{2\alpha}{k_{i-1}k_{i-3}k_{i-5}\cdots\over k_ik_{i-2}k_{i-4}\cdots}.
\]
After extension by periodicity, since the number of square in $\Omega$ is bounded by ${Ch^{-2\alpha}}$, we obtain for some $C>0$
\begin{equation}\label{estim1}
\left\vert \left\{x\in \Omega:\nabla v_{k_2,\dots,k_i}(x)\notin \Sigma\right\} \right\vert\le C\left(\displaystyle \sum_{i=2}^{L}{\displaystyle \prod_{l=0}^{1+{i\over 2}}k_{i-2l-1}\over \displaystyle \prod_{l=0}^{i\over 2}k_{i-2l}}\right).
\end{equation}
\subsubsection*{Error estimation due to the interpolation} After interpolation, we have a bound due to the interpolation as
\[
C h{h^{\alpha}\over \displaystyle \prod_{l=0}^{1+{i\over 2}}k_{i-2l-1}}k_i\times N_s,
\]
where $N_s$ is the number of these cells in each rectangle $Q_s$. We deduce that the bound is
\[
Ch^{1+\alpha}{k_{i}k_{i-2}\cdots}=Ch^{1+\alpha}\prod_{l=0}^{i\over 2}k_{i-2l}.
\]
We have to take into account of the error due to the interpolation for $v^{(1)}$, so we set $k_0=k_1=1$. After extension by periodicity, we obtain a bound 
\begin{equation}\label{estim2}
Ch^{1-\alpha}\sum_{i=1}^L {\prod_{l=0}^{i\over 2}k_{i-2l}}
\end{equation}
since the number of square is bounded by ${Ch^{-2\alpha}}$.
\subsubsection*{Error estimation at the boundary $\partial\Omega$}  The bound is 
\begin{equation}\label{estim3}
Ch^\alpha,
\end{equation}
see Subsection~\ref{L=2}~\ref{error-boundary-estimate} for details.

Since $\nabla \widetilde{w}_{k_2,\dots,k_L}$ is bounded independently of $h$, it follows $u^\alpha_{k_2,\dots,k_L}\in V_{0,C}^h$ for some $C>0$.  

\subsection{End of the proof of Theorem ~\ref{main result}} 
Collecting the estimations (\ref{estim1}), (\ref{estim2}) and (\ref{estim3}), we obtain for some $C>0$
\begin{align}\label{estim_totale1}\Err^\Sigma\left(u^\alpha_{k_2,\dots,k_L}\right)\le C\left(h^\alpha+h^{1-\alpha}\sum_{i=1}^L {\displaystyle \prod_{l\in I_{i\over 2}}k_{i-2l}}+\sum_{i=2}^{L}{\displaystyle \prod_{l\in I_{i-1\over 2}}k_{i-2l-1}\over \displaystyle\prod_{l\in I_{i\over 2}}k_{i-2l}}\right).
\end{align}
\begin{lemma}\label{induction1} For every $i\in\{1,\dots,L-2\}$ there exists $C_i>0$ such that for every $(k_2,\dots,k_{L-i+1})\in\prod_{j=2}^{L-i+1}\A_j$, there exists $h_i>0$ so that for every $h\in]0,h_i[$ there exists $(\widehat{k}_{L-i+2},\dots,\widehat{k}_{L})\in\prod_{j=L-i+2}^{L}\A_j$ satisfying 
\begin{align}
&\Err^\Sigma\left(u^\alpha_{k_2,\dots,{k}_{L-i+1},\widehat{k}_{L-i+2},\dots,\widehat{k}_L}\right)\notag\\
 \le& \displaystyle C_i\left(h^\alpha+h^{1-\alpha}\sum_{j=1}^{L-i} {\prod_{l=0}^{j\over 2}k_{j-2l}} + \displaystyle\sum_{j=2}^{L-i}{\displaystyle\prod_{l=0}^{{j-1\over 2}}k_{j-2l-1}\over\displaystyle\prod_{l=0}^{j\over 2}k_{j-2l}}\displaystyle+\left(h^{1-\alpha}\prod_{l=0}^{{L-i+1\over 2}}k_{L-i+1-2l}\right)^{1\over i}\right.\notag\\
 &\hspace{8.cm}\left.+{\displaystyle\prod_{l=0}^{{L-i\over 2}}k_{L-i-2l}\over\displaystyle\prod_{l=0}^{L-i+1\over 2}k_{L-i+1-2l}}\right).\label{inf_estimat_totale1}
\end{align}
\end{lemma}

Without loss of generality we assume that $L\ge 3$. From Lemma~\ref{induction1}, we deduce that there exists $C_{L-2}>0$ such that for every $(k_2,k_3)\in\A_2\times\A_3$ there exists $h_{L-2}>0$ satisfying
\begin{eqnarray*}
\Err^\Sigma(u^\alpha_{k_2,k_3,\widehat{k}_4,\dots,\widehat{k}_L}) &\le &  C_{L-2}\left(h^\alpha+{1\over k_2}+{k_2\over k_3}+h^{1-\alpha}(1+k_2)+({k_3}h^{1-\alpha})^{{1\over L-2}}\right)
\end{eqnarray*} 
for all $h\in]0,h_{L-2}[$. 

The function $]0,+\infty[\ni x\mapsto \sigma_{L-2}(x):={k_2\over x}+({x}h^{1-\alpha})^{{1\over L-2}}$ reaches its minimum at $x_3=(L-2)^{L-2\over L-1}{k_2}^{L-2\over L-1}h^{-{1-\alpha\over L-1}}$. Set $\widehat{k}_3:=\ceil [\big]{x_3}+1$. We have
\begin{align*}
&\inf_{k_3\in\NN}\sigma_{L-2}\left(k_3\right)\\
\le& \sigma_{L-2}\left(\widehat{k}_3\right)\\
\le& (L-2)^{-\frac{L-2}{L-1}}\left(k_2h^{1-\alpha}\right)^{\frac{1}{L-1}}+2^{L-3}\left((L-2)^{\frac{1}{L-1}}\left(k_2h^{1-\alpha}\right)^{\frac{1}{L-1}}+h^{\frac{1-\alpha}{L-2}}\right)\\
\le&\left((L-2)^{-\frac{L-2}{L-1}}+2^{L-3}\left(1+(L-2)^{\frac{1}{L-1}}\right)\right)\left(k_2h^{1-\alpha}\right)^{\frac{1}{L-1}}
\end{align*}
since $k_2\ge 1\ge h^{\frac{1}{L-2}}$. In order to have $\widehat{k}_3\in\A_3$ we take $h<h^\prime$ with
\begin{align*}
h^\prime:=\left((L-2)k_2\right)^{\frac{L-2}{1-\alpha}} \left(\c k_2-1\right)^{-{L-1\over 1-\alpha}}
\end{align*}
Now, if $h<h^\second:=\left(\frac{1}{k_2}\right)^{\frac{1}{1-\alpha}}$ we have
\begin{align}\label{eq-cl-2}
h^{1-\alpha}(1+k_2)\le 2\left(k_2h^{1-\alpha}\right)^{\frac{1}{L-1}}
\end{align}
since $k_2+1\le 2k_2$. 

We set
\begin{align*}
C_{L-1}:=&\max\left\{2,\left((L-2)^{-\frac{L-2}{L-1}}+2^{L-3}\left(1+(L-2)^{\frac{1}{L-1}}\right)\right),C_{L-2}\right\}\\
h_{L-1}:=&\min\left\{h_{L-2}, h^\prime,h^\second\right\}.
\end{align*}
Taking \eqref{eq-cl-2} into account we obtain
\begin{eqnarray*}
\Err^\Sigma(u^\alpha_{k_2,\widehat{k}_3,\widehat{k}_4,\dots,\widehat{k}_L})
\le C_{L-1}\left( h^\alpha+{1\over k_2}+\left(k_2h^{1-\alpha}\right)^{1\over L-1}\right).
\end{eqnarray*}
for all $h\in ]0,h_{L-1}[$.

The function $]0,+\infty[\ni x\mapsto \sigma_{L-1}(x):={1\over x}+(xh^{1-\alpha})^{1\over L-1}$ is minimum for $x_2=h^{-{1-\alpha\over L}}$. Set $\widehat{k}_2:=\ceil [\big]{x_2}+1$. Thus, similarly as above, there exist $C_L>0$ and $h_L>0$, such that we have 
\[
\Err^\Sigma(u^\alpha_{\widehat{k}_2,\widehat{k}_3,\widehat{k}_4,\dots,\widehat{k}_L})
\le C_L \left(h^\alpha+h^{1-\alpha\over L}\right)
\]
for all $h\in ]0,h_L[$.
The best constant $\alpha$ is $\widehat{\alpha}={1-\widehat{\alpha}\over L}$, i.e., $\widehat{\alpha}={1\over 1+L}$. Set $\widehat{u}:=u^{\widehat\alpha}_{\widehat{k}_2,\widehat{k}_3,\widehat{k}_4,\dots,\widehat{k}_L}$, we obtain
\[
\Errm^K\le\Errm^\Sigma\le\Err^\Sigma(\widehat u)\le C_L h^{1\over 1+L}.
\]
Finally we can adjust $C_L$ in order to have $\widehat u\in V_{0,C_L}^h$ at the same time. The proof is complete.
\hfill$\blacksquare$
\subsubsection*{Proof of Lemma~\ref{induction1}} The proof follows by induction. Let $i=1$ then (\ref{inf_estimat_totale1}) corresponds to (\ref{estim_totale1}). Take $C_1=C$, $h_1$ defined in Subsection~\ref{h1-subsection} (1), and any $\widehat{k}_L\in \A_{L}$. 

Now, fix $i\in\{1,\dots,L-2\}$ and choose the corresponding $C_i>0$ such that for every $(k_2,\dots,k_{L-i+1})\in\prod_{j=2}^{L-i+1}\A_j$, there exists $h_i>0$, so that for every $h\in ]0,h_i[$ we can find $(\widehat{k}_{L-i+2},\dots,\widehat{k}_{L})\in\prod_{j=L-i+2}^{L}\A_j$ satisfying \eqref{inf_estimat_totale1}.

Fix $(k_2,\dots,k_{L-i})\in \prod\limits_{2\le j\le L-i+1}\A_j$. 
Set 
\begin{align*}
\theta_i:=\prod_{l=0}^{{L-i\over 2}}k_{L-i-2l}\quad\mbox{ and }\quad\gamma_i:=\prod_{l=0}^{L-i-1\over 2}k_{L-i-1-2l}.
\end{align*}
Fix $h\in ]0,1[$. Minimize over $\A_{L-i+1}$ the left hand side of the inequality \eqref{inf_estimat_totale1}. For this we  consider the function $x\mapsto\sigma_{i}(x)=\left(h^{1-\alpha}x\gamma_i\right)^{1\over i}+{\theta_i\over \displaystyle x\gamma_i}$ which reaches its minimum at 
\[
x_i=i^{i\over i+1}h^{-{1-\alpha\over i+1}}{\theta_i^{i\over i+1}\over\gamma_i}.
\] 
We set $\widehat{k}_{L-i+1}:=\ceil [\big]{x_i}+1$. We can see that $\sigma\left(\widehat{k}_{L-i+1}\right)\le \sigma(x_i)+\left(h^{1-\alpha}\gamma_i\right)^{1\over i}$. If $h<h^\prime:=\left(\theta_i^i\over\gamma_i^{i+1}\right)^{1-\alpha}$ then $\left(h^{{1-\alpha}}\theta_i\right)^{1\over i+1}>\left(h^{{1-\alpha}}\gamma_i\right)^{1\over i}$. So, for $c_i:=i^{1\over i+1}+i^{-{i\over i+1}}+1$ we have
\begin{align}\label{sigmai-integer}
\inf_{{k}_{L-i+1}\in\NN}\sigma({k}_{L-i+1})\le\sigma\left(\widehat{k}_{L-i+1}\right)\le c_i \left(h^{{1-\alpha}}\theta_i\right)^{1\over i+1}.
\end{align}
The condition $\widehat{k}_{L-i+1}\in\A_{L-i+1}$ is equivalent to $h<h^\second$ with
\[
h^\second:=\left(\gamma_i(\c k_{L-i}-1)\right)^{-\frac{i+1}{1-\alpha}}\left(i\theta_i\right)^{\frac{i}{1-\alpha}}.
\]
Let $C^\prime:=\max\left\{C_i, c_i\right\}$. Then we obtain that for every $h\in ]0,\min\{h^\prime,h^\second,h_{i}\}[$
\begin{align*}
&\Err^\Sigma\left(u^\alpha_{k_2,\dots,\widehat{k}_{L-i+1},\dots,\widehat{k}_L} \right)\\
\le&   \displaystyle C^\prime\left(h^\alpha+h^{1-\alpha}\sum_{j=1}^{L-i} {\displaystyle\prod_{l=0}^{j\over 2}k_{j-2l}}+\sum_{j=2}^{L-i}{\displaystyle\prod_{l=0}^{{j-1\over 2}}k_{j-2l-1}\over\displaystyle\prod_{l=0}^{j\over 2}k_{j-2l}}+\left(h^{1-\alpha}\displaystyle\prod_{l=0}^{{L-i\over 2}}k_{L-i-2l}\right)^{1\over i+1}\right).
\end{align*}
Note that if $h<\theta_i^{-{1\over 1-\alpha}}$ then
\[
\left(h^{1-\alpha}\prod_{l=0}^{{L-i\over 2}}k_{L-i-2l}\right)^{1\over i+1}>h^{1-\alpha}\prod_{l=0}^{{L-i\over 2}}k_{L-i-2l}.
\]
Set $h_{i+1}:=\min\left\{\theta_i^{-{1\over 1-\alpha}},h^\prime,h^\second,h_i\right\}$ and $C_{i+1}:=\max\{C^\prime,2\}$, then for every $h\in ]0,h_{i+1}[$
\begin{align}
&\Err^\Sigma\left(u^\alpha_{k_2,\dots,\widehat{k}_{L-i+1},\dots,\widehat{k}_L} \right)\notag\\ 
\le&   C_{i+1}\left(h^\alpha+h^{1-\alpha}\sum_{j=1}^{L-i-1} {\displaystyle\prod_{l=0}^{j\over 2}k_{j-2l}}+\sum_{j=2}^{L-i-1}{\displaystyle\prod_{l=0}^{{j-1\over 2}}k_{j-2l-1}\over\displaystyle\prod_{l=0}^{j\over 2}k_{j-2l}}\right.\notag\\
&\hspace{2cm}+ \left.\left(h^{1-\alpha}\displaystyle\prod_{l=0}^{{L-i\over 2}}k_{L-i-2l}\right)^{1\over i+1}+{\displaystyle\prod_{l=0}^{{L-i-1\over 2}}k_{L-i-2l-1}\over\displaystyle\prod_{l=0}^{L-i\over 2}k_{L-i-2l}}\right).\label{inf_estimat_totale2}
\end{align}
The proof is complete.
\hfill$\blacksquare$

\section*{Acknowledgement} The author thanks Michel Chipot for the suggestion of this problem. A part of this work was developed during a temporary position at the Institute of Mathematics, University of Z\"urich.

\end{document}